\documentclass[a4paper,11pt]{article}

\usepackage[utf8]{inputenc}
\DeclareUnicodeCharacter{00A0}{~}

\usepackage[english]{babel}

\usepackage{a4wide}
\usepackage{amssymb,amsmath,amscd,amsfonts,amsthm,bbm,mathrsfs,enumerate}
\usepackage[colorlinks=true, linkcolor=blue, citecolor=blue]{hyperref}
\usepackage{graphicx}
\usepackage{tikz,pgfplots}
\usepackage{color}
\usepackage{csquotes}
\usepackage{authblk}

\usepackage{shortcuts}

\newcommand{\bs}{\boldsymbol}
\newcommand{\1}{\mathbbm 1}
\renewcommand{\eqdef}{:=}

\newcommand{\NN}{\mathbb N}
\newcommand{\RR}{\mathbb R}

\newcommand{\PP}{\mathbb P}
\newcommand{\EE}{\mathbb E}
\newcommand{\SSS}{\mathbb S}

\newcommand{\mC}{{\mathcal C}}

\newcommand{\mK}{{\mathcal K}}

\newcommand{\ps}[1]{\langle #1 \rangle}

\newcommand{\mcB}{{\mathscr B}}
\newcommand{\mcT}{{\mathscr T}}
\newcommand{\mcF}{{\mathscr F}}

\newcommand{\sH}{{\mathsf H}}
\newcommand{\sx}{{\mathsf x}}
\newcommand{\sy}{{\mathsf y}}
\newcommand{\sX}{{\mathsf X}}
\newcommand{\Pg}{P_\gamma}

\theoremstyle{plain} 

\newtheorem{theorem}{Theorem}
\newtheorem{lemma}{Lemma}

\newtheorem{definition}{Definition}
\newtheorem{proposition}{Proposition}
\newtheorem{assumption}{Assumption}

\newtheorem{example}{Example}

\makeatletter
\newcommand{\neutralize}[1]{\expandafter\let\csname c@#1\endcsname\count@}
\makeatother

\DeclareMathOperator*{\support}{supp}

\DeclareMathOperator{\zer}{zer}
\DeclareMathOperator{\inter}{int}

\DeclareMathOperator{\cl}{cl}
\DeclareMathOperator{\dist}{dist}

\title{Convergence of constant step stochastic gradient descent \\
   for non-smooth non-convex functions}
\author[1]{Pascal Bianchi}
\author[2]{Walid Hachem}
\author[2]{Sholom Schechtman}

\affil[1]{LTCI, Telecom Paris, IP Paris, France.}
\affil[2]{LIGM, CNRS, Univ Gustave Eiffel, ESIEE Paris, F-77454 Marne-la-Vall\'ee, France.}

\date{January 2021}

\begin{document}

\maketitle

\begin{abstract}
  This paper studies the asymptotic behavior
  of the constant step Stochastic Gradient Descent for the minimization
  of an unknown function, defined as the expectation of a
  non convex, non smooth, locally Lipschitz random function.
  As the gradient may not exist, it is replaced by a certain operator:
  a reasonable choice is to use an element of the Clarke subdifferential of the random function;
  another choice is the output of the celebrated backpropagation algorithm, which is popular
  amongst practioners, and whose properties have recently
  been studied by Bolte and Pauwels.
  Since the expectation of the chosen operator is not in general an element of the Clarke subdifferential
  of the mean function,
  it has been assumed in the literature that an oracle of the Clarke subdifferential of the mean function
  is available. As a first result, it is shown
  in this paper that such an oracle is not needed for almost all initialization
  points of the algorithm.
  Next, in the small step size regime, it is shown that  the
  interpolated trajectory of the algorithm converges in probability (in the compact convergence sense)
  towards the set of solutions of a particular differential inclusion: the subgradient flow.
  Finally, viewing the iterates as a Markov chain whose transition kernel is indexed by the
  step size, it is shown that the invariant distribution of the kernel converge weakly
  to the set of invariant distribution of this differential inclusion as the step size tends to zero.
  These results show that when the step size is small, with large probability, the
  iterates eventually lie in a neighborhood of the critical points of the
  mean function.

%

\end{abstract}


\section{Introduction}
\label{intro}

In this work, we study the asymptotic behavior of the constant step
Stochastic Gradient Descent (SGD) when the objective function is
neither differentiable nor convex.  Given an integer $d\geq 1$ and a
probability space $(\Xi, \mcT, \mu)$, let
$f : \RR^d \times \Xi \to \RR, (x,s) \mapsto f(x,s)$ be a function
which is assumed to be locally Lipschitz, generally non-differentiable
and non-convex in the variable $x$, and $\mu$-integrable in the
variable $s$. The goal is to find a local minimum, or at least a
critical point of the function
$F(x) = \int f(x, s) \, \mu(ds) = \EE f(x, \cdot)$, \emph{i.e.}, a
point $x_\star$ such that $0 \in \partial F(x_\star)$, where
$\partial F$ is the so-called Clarke subdifferential of $F$.  It is
assumed that the function $f$ is available to the observer along with
a sequence of independent $\Xi$-valued random variables
$(\xi_k)_{k\in\NN}$ on some probability space with the same
probability law $\mu$.  The function $F$ itself is assumed unknown due
to, \emph{e.g.}, the difficulty of computing the integral
$\EE f(x,\cdot)$.  Such non-smooth and non-convex problems are
frequently encountered in the field of statistical learning. For instance this
type of problem arises in the study of neural networks when the activation function
is non-smooth, which is the case of the commonly used ReLU function.

We establish the weak convergence of SGD to the set of (Clarke) critical points of $F$.
Our main contributions are:
\begin{itemize}
\item We investigate the constant step size regime, whereas most works address the vanishing step size regime.
\item We study an \emph{oracle-free} version of SGD, which does not require to have access to the
  Clarke subgradient of the unknown function $F$.
\end{itemize}
To that end, our main hypotheses is that the function $F$ is \emph{Whitney stratifiable}.
We also need to posit that the sequence of iterates is \emph{bounded in probability}.
Boundedness assumptions are quite standard in stochastic approximation, we nevertheless provide sufficient conditions:
first, it holds when $F$ is assumed coercive and smooth outside an arbitrary compact set; second, it naturally holds
in the case of \emph{projected SGD} \emph{i.e.}, when the iterates are projected onto some compact set.
The convergence of the projected SGD is as well addressed in the paper.

We say that a sequence of random variables $(x_n)_{n\in \NN}$ on $\RR^d$
is a \emph{SGD sequence} with step size $\gamma>0$ if, with probability one,
\begin{equation}
  \label{eq:sgd}
  x_{n+1} = x_n - \gamma\nabla f(x_n,\xi_{n+1})
\end{equation}
for every $n$ such that the function $f(\cdot,\xi_{n+1})$ is
differentiable at point $x_n$, where $\nabla f(x_n,\xi_{n+1})$
represents the gradient w.r.t. the variable $x_n$. When
$f(\cdot,\xi_{n+1})$ is non-differentiable at  $x_n$, the
update equation $x_n\to x_{n+1}$ is left undefined.
The practioner is free to choose the value of $x_{n+1}$ according to a
predetermined selection policy. Typically, a reasonable choice
is to select $x_{n+1}$ in the set $x_n-\gamma \partial
f(x_n,\xi_{n+1})$, where $\partial f(x, s)$ represents the Clarke
subdifferential of the function $f(\cdot, s)$ at the point $x$.  When
such a policy is used, the resulting sequence will be referred to as a
\emph{Clarke-SGD} sequence.
In fact, our study extends to the case where $x_{n+1}$ is chosen in the set $x_n - \gamma G_{f(\cdot, \xi_{n+1})}$,
where $G_{f(\cdot, \xi_{n+1})}$ is a generalized subdifferential of $f(\cdot, \xi_{n+1})$
in Norkin's sense \cite{nor_gen80} (we refer to such a sequence as a \emph{Norkin-SGD} sequence).
The Clarke subdifferential is a special case of generalized subdifferential.

An alternative used by practioners is to
compute the derivative using the automatic differentiation provided in popular
API's such as Tensorflow, PyTorch, etc. \emph{i.e.}, for all $n$,
\begin{equation}
x_{n+1} = x_n - \gamma a_{f(\cdot,\xi_{n+1})}(x_n)\label{eq:autograd}
\end{equation}
where $a_h$ stands for the output of the automatic differentiation applied to a function $h$.
We refer to such a sequence as an \emph{autograd} sequence.
This approach is useful when $f(\cdot,s)$ is a composition of matrix multiplications and
non-linear activation functions, of the form
\begin{equation}
\label{cnn}
f(x,s) =
  \ell(\sigma_L(W_L\sigma_{L-1}(W_{L-1} \cdots \sigma_1(W_1 X_s))),Y_s)\,,
\end{equation}
where $x=(W_1,\cdots,W_L)$ are the weights of the network represented
by a finite sequence of $L$ matrices, $\sigma_1,\cdots,\sigma_L$ are
vector-valued functions, $X_s$ is a feature vector, $Y_s$ is a label
and $\ell(\cdot,\cdot)$ is some loss function.  In such a case, the
automatic differentiation is computed using the chain rule of function
differentiation, by means of the celebrated backpropagation algorithm.
When the mappings $\sigma_1,\cdots,\sigma_L,\ell(\cdot,Y_s)$ are differentiable, the
chain rule indeed applies and the output coincides with the gradient.
However, the chain rule fails in case of non-differentiable functions.
The properties of the map $a_h$ are studied in the recent work
\cite{bolte2019conservative}.
In general, $a_h(x)$ may not be an element of the Clarke-subdifferential
$\partial h(x)$. It can even happen that $a_h(x)\neq \nabla h(x)$ at some
points $x$ where $h$ is differentiable. However, the set of such peculiar
points is proved to be Lebesgue negligible. As a consequence, if the initial
point $x_0$ is chosen random according to some density w.r.t. the Lebesgue
measure, an autograd sequence can be shown to be a SGD sequence in the sense of
Equation~(\ref{eq:sgd}) under some conditions.
The aim of this paper is to analyze the asymptotic behavior of SGD sequences
in the case where the step $\gamma$ is constant.

\smallskip

{\bf About the literature.} In the nonsmooth and non convex case, the
convergence of SGD has been studied in \cite{ermoliev2003solution} and \cite{erm-nor-98}
using the concept of generalized differentiability \cite{nor_gen80},
and assuming a Sard-like condition on the critical set.
More recently, using a differential inclusion (DI) approach, the papers \cite{dav-dru-kak-lee-19}
provide a similar result under the additional assumption that the objective function
is Whitney-stratifiable (see also \cite{maj-mia-mou18}, in the particular case of
subdifferentially regular functions). These papers make two major hypotheses on
the algorithm under study, which we avoid in this paper.

The first major hypothesis in the above papers if the fact that the step size
is vanishing, \emph{i.e.}, $\gamma$ is replaced with a sequence $(\gamma_n)$ that tends to zero
as $n\to \infty$.
From a theoretical point of view, the vanishing step size is convenient
because, under various assumptions,
it allows to demonstrate the almost sure convergence of the iterates
$x_n$ to the set
\begin{equation}
\cZ \eqdef \{ x \in \RR^d \, : \, 0 \in \partial F(x) \} \label{eq:S}
\end{equation}
of critical points of $F$. However, in practical applications such as
neural nets, a vanishing step size is rarely used because of slow
convergence issues. In most computational frameworks,
a possibly small but nevertheless constant step size is used by default.
The price to pay is that the iterates are no longer expected to converge
almost surely to the set $\cZ$ but to
fluctuate in the vicinity of $\cZ$ as $n$ is large. In this paper, we aim at
establishing a result of the type
\begin{equation}
\label{eq:main_res}
  \forall \varepsilon > 0, \quad \limsup_{n \to \infty}
  \bbP(\bs d(x_n, \cZ) > \varepsilon) \xrightarrow[\gamma \downarrow 0]{} 0,
\end{equation}
where $\bs d$ is the Euclidean distance between $x_n$ and the set $\cZ$.
Although this result is weaker than in the vanishing step case, constant step
stochastic algorithms can reach a neighborhood of $\cZ$ faster than their
decreasing step analogues, which is an important advantage in the applications
where the accuracy of the estimates is not essential.  Moreover, in practice
they are able to cope with non stationary or slowly changing environments which
are frequently encountered in signal processing, and possibly track a changing
set of solutions \cite{ben-met-pri-livre90,kus-yin-(livre)03}.

The second important difference between the present paper and the papers
\cite{maj-mia-mou18,dav-dru-kak-lee-19} lies in the algorithm under study.
In these papers, the iterates are supposed
to satisfy the inclusion
\begin{equation}
\frac{x_{n+1}- x_n}{\gamma_{n+1}} \in - \partial F(x_n)+  \eta_{n+1}\,\label{eq:oracle}
\end{equation}
for all $n$, where $(\eta_{n})$ is a martingale increment noise w.r.t. the filtration
$(\sigma(x_0, \xi_1, \ldots, \xi_n))_{n\geq 1}$.
Under the assumption that $\gamma_n \to 0$ as $n\to\infty$,
the authors of \cite{maj-mia-mou18,dav-dru-kak-lee-19} prove that almost surely,
the continuous time linearly interpolated process constructed from a sequence $(x_n)$ satisfying (\ref{eq:oracle})
is a so-called asymptotic pseudotrajectory \cite{ben-hof-sor-05} of the
Differential Inclusion (DI)
\begin{equation}
\dot \sx(t) \in - \partial F(\sx(t))\,,
\label{eq:DI}
\end{equation}
that will be defined on $\RR_+ = [0,\infty)$.
Heuristically, this means that a sequence $(x_n)$ satisfying~(\ref{eq:oracle})
shadows a solution to (\ref{eq:DI}) as $n$ tends to infinity. This result is
one of the key ingredients to establish the almost sure convergence of $x_n$ to
the set $\cZ$.  Unfortunately, a SGD sequence does not satisfy the
condition~(\ref{eq:oracle}) in general (setting apart the fact that $\gamma$ is
constant).  To be more precise, consider a Clarke-SGD sequence as defined
above.  For all $n$, $x_{n+1} = x_n -\gamma \partial f(x_n,\xi_{n+1})$, which
in turn implies
\begin{align*}
\frac{x_{n+1}- x_n}\gamma &\in - \EE\partial f(x_n,\,.\,) + \eta_{n+1}\,,
\end{align*}
where $(\eta_n)$ is a martingale increment noise sequence, and where $\EE\partial f(x,\,.\,)$
represents the set-valued expectation $\int \partial f(x,s)d\mu(s)$.
The above inclusion is analogous to~(\ref{eq:oracle}) in the case where
$\partial F(x) = \EE \partial f(x, \cdot)$ for all $x$ \emph{i.e.},
if one can interchange the expectation $\EE$ and the Clarke subdifferential operator $\partial$.
Although the interchange
holds if \emph{e.g.}, the functions $f(\cdot, s)$ are convex (in which
case $\partial f(x,s)$ would coincide with the classical convex
subdifferential), one has in general
$\partial \EE f(x, \cdot) \subset \EE \partial f(x, \cdot)$ and the
inclusion can be strict \cite[Proposition
2.2.2]{cla-led-ste-wol-livre98}.
As a consequence, a Clarke-SGD sequence does not admit the oracle form
(\ref{eq:oracle}) in general. For such a sequence, the corresponding DI reads
\begin{equation}
\dot \sx(t) \in - \EE \partial f(\sx(t),\,.\,)\,,
\end{equation}
but unfortunately, the flow of this DI may contain spurious equilibria
(an example is provided in the paper).
In \cite{maj-mia-mou18} the authors restrict their analysis to \emph{regular} functions \cite[\S
2.4]{cla-led-ste-wol-livre98}, for which the interchange of the expectation and
the subdifferentiation applies.  However, this assumption can be restrictive, since a
function as simple as $- | x|$ is not regular at the critical point zero.
The issue of the absence of interchange between the expectation and the Clarke subdifferential
was addressed in \cite{erm-nor-98} using the notion of generalized differentiability.
In this work, the convergence is established towards the set of zeroes of the
generalized subdifferential of $F$. However, this set can be substantially larger than
the set $\cZ$ of critical points.

A second example where the oracle form of Equation~(\ref{eq:oracle}) does not hold
is given by autograd sequences. Such an example is studied in
\cite{bolte2019conservative}, assuming that the step size is vanishing and that
$\xi$ takes its values over a finite set. It is proved that the autograd
sequence is an almost sure asymptotic pseudotrajectory of the DI $\dot
\sx(t)\in - D(\sx(t))$, for some set-valued map $D$ which is shown to be a
\emph{conservative} field with $F$ as a potential. Properties of conservative
fields are studied in \cite{bolte2019conservative}. In particular, it is proved
that $D=\{\nabla f\}$ Lebesgue almost everywhere. Despite this property, the DI
$\dot \sx(t)\in - D(\sx(t))$ substantially differs from (\ref{eq:DI}).  Again,
the set of equilibria may be strictly larger than the set $\cZ$ of
critical points of $F$.

We finally mention the paper~\cite{rusz_20}, which studies an
inertial version of SGD in the vanishing step size regime.
Similarly to \cite{maj-mia-mou18,dav-dru-kak-lee-19} and contrary to the present paper,
the author assumes the oracle form of Equation~(\ref{eq:oracle}).
The almost sure convergence is established, under the rather weak assumption that
$F$ is differentiable in Norkin's generalized sense.

\subsubsection*{Contributions}

\begin{itemize}
\item We analyze the SGD algorithm~(\ref{eq:sgd}) in the non-smooth,
  non-convex setting, under realistic assumptions: the step size is
  assumed to be constant along the iterations, and we neither assume
  the regularity of the functions involved, nor the knowledge of an
  oracle of $\partial F$ as in (\ref{eq:oracle}).
  Our assumptions encompass Clarke SGD sequences, autograd and Norkin SGD sequences
  as special cases.

\item Under mild conditions, we prove that when the initialization $x_0$ is
randomly chosen with a density, all SGD sequences
coincide almost surely, irrespective to the particular selection policy used at the
points of non-differentiability. In this case, $x_n$ almost never hits
a non-differentiable point of $f(\cdot,\xi_{n+1})$ and Equation~(\ref{eq:sgd}) actually holds for all $n$.
Moreover, we prove that
$$
\frac{x_{n+1}- x_n}\gamma = - \nabla F(x_n) + \eta_{n+1}\,,
$$
where $(\eta_n)$ is a martingale difference sequence, and $\nabla F(x_n)$ is the true gradient
of $F$ at $x_n$. This argument allows to bypass the oracle assumption of
\cite{maj-mia-mou18,dav-dru-kak-lee-19}.
\item We establish that the continuous process obtained by piecewise affine interpolation
of $(x_n)$ is a \emph{weak asymptotic pseudotrajectory} of the DI~(\ref{eq:DI}).
In other words, the interpolated process converges in probability to the set
of solutions to the DI, as $\gamma\to 0$, for the metric of uniform convergence on compact intervals.
\item We establish the long run convergence of the iterates $x_n$ to the set
$\cZ$ of Clarke critical points of $F$, in the sense of Equation~(\ref{eq:main_res}).
This result holds under two main assumptions. First, it assumed that
$F$ admits a chain rule, which is satisfied for instance if $F$ is a so-called tame function.
Second, we assume a standard drift condition on the Markov chain~(\ref{eq:sgd}).
Finally, we provide verifiable conditions of the functions $f(\cdot,s)$
under which the drift condition holds.

\item In many practical situations, the drift conditions alluded to above are
not satisfied.  To circumvent this issue, we analyze a projected version of the
SGD algorithm, which is similar in its principle to the well-known projected
gradient algorithm in the classical stochastic approximation theory.

\end{itemize}

\subsubsection*{Paper organization}
Section~\ref{subsec:prel} recalls some known facts about Clarke
subdifferentials, conservative fields and differential inclusions. In
Section~\ref{sec:ae-gradient}, we study the elementary properties of
almost-everywhere gradient functions, defined as the functions $\varphi(x,s)$
which coincide with $\nabla f(x,s)$ almost everywhere. Practical examples are
provided. In Section~\ref{sec:sgd}, we study the elementary properties of SGD
sequences. Section~\ref{sec:dyn_beh} establishes the convergence in probability
of the interpolated process to the set of solutions to the DI. In
Section~\ref{sec:longrun}, we establish the long run convergence of the
iterates to the set of Clarke critical points.  Section~\ref{sec:proj-sgd} is
devoted to the projected subgradient algorithm.  The proofs are found in
Section~\ref{sec:proofs}.

\section{Preliminaries}\label{subsec:prel}

\subsection{Notations}\label{sec:not}

If $\nu,\nu'$ are two measures on some measurable space $(\Omega,\cF)$,
$\nu\ll\nu'$ means that $\nu$ is absolutely continuous w.r.t. $\nu$.
The $\nu$-completion of $\mcF$ is defined as the sigma-algebra consisting of
the sets $S \subset \Omega$ such that there exist $A,B \in \mcF$ with
$A \subset S \subset B$ and $\nu(B\setminus A) = 0$. For these sets,
$\nu(S) = \nu(A)$.

If $E$ is a metric space, we denote by $\mcB(E)$ the Borel sigma field on $E$.
Let $d$ be an integer. We denote by $\cM(\RR^d)$ the set of probability
measures on $\mcB(\RR^d)$ and by $\cM_1(\RR^d)\eqdef\{\nu\in
\cM(\RR^d):\int\|x\|\nu(dx)<\infty\}$. We denote as $\lambda^d$ the Lebesgue
measure on $\RR^d$. When the dimension is clear from the context, we denote
as $\lambda$ this Lebesgue measure. For a subset $\cK\subset\RR^d$, we
denote by
$$
\cM_{abs}(\cK)\eqdef\{\nu\in \cM(\RR^d):\nu\ll\lambda\text{ and }\text{supp}(\nu)\subset\cK\}\,,
$$
where $\text{supp}(\nu)$ represents the support of $\nu$.

If $P$ is a Markov kernel on $\bbR^d$ and
$g:\bbR^d\to\bbR$ is a measurable function, $Pg$ represents the function on
$\bbR^d\to\bbR$ given by $Pg(x) = \int P(x, \dif y)g(y)$, whenever the integral
is well-defined (the integral is understood in the weak sense). For every measure $\pi\in {\mathcal M}(\bbR^d)$, we denote by
$\pi P$ the measure given by $\pi P = \int \pi(\dif x) P(x, \cdot)$.  We use
the notation $\pi(g) = \int g \dif \pi$ whenever the integral is well-defined.

For every $x\in \RR^d$, $r>0$, $B(x,r)$ is the open
Euclidean ball with center $x$ and radius $r$.
The notation $\1_A$ stands for the indicator function of a set $A$, equal to
one on that set and to zero otherwise. The notation $A^c$ represents the
complementary set of a set $A$ and $\cl(A)$ its closure.

\subsection{Subdifferentials and Conservative Fields}


\label{sec:clark_cons_field}

A set valued map $H : \bbR^d \rightrightarrows \bbR^d$ is a map such that for
each $x \in \bbR^d$, $H(x)$ is a subset of $\bbR^d$.  We say that $H$ is upper semi continuous, if its graph $\{ (x, y): y \in H(x)\}$ is closed in $\bbR^{d \times d}$. For any function
$F:\bbR^d \rightarrow \bbR$, we denote by $\cD_F$ the set of points $x\in
\RR^d$ such that $F$ is differentiable at $x$. If $F$ is locally Lipschitz
continuous, it is by Rademacher's theorem almost everywhere differentiable. In
this case, the Clarke's subdifferential of $F$ coincides with the set-valued
map $\partial F:\bbR^d\rightrightarrows\bbR^d$ given for all $x\in \bbR^d$ by
\begin{equation*}
  \partial F (x) = \co \left\{ y\in \bbR^d\,:\,\exists (x_n)_{n\in \NN}\in \cD_F^\NN\ \text{ s.t. } (x_n,\nabla F(x_n)) \rightarrow (x,y)\right\}\,,
\end{equation*}
where $\co$ stands for the convex hull \cite{cla-led-ste-wol-livre98}.

We now briefly review some recent results of
\cite{bolte2019conservative}.  A set-valued map
$D:\bbR^d\rightrightarrows \bbR^d$ is called a \emph{conservative field}, if
for each $x \in \bbR^d$, $D(x)$ is a nonempty and compact subset of
$\bbR^d$, $D$ has a closed graph, and for each absolutely continuous
$a \colon [0, 1] \to \bbR^d$, with $a(0) = a(1)$, it holds that:
 \begin{equation*}
   \int_{0}^1 \min_{v \in D(a(t))} \scalarp{\dot a(t)}{v} \dif t= \int_{0}^1 \max_{v \in D(a(t))}\scalarp{\dot a(t)}{v} \dif t = 0 \, .
 \end{equation*}
We say that a function $F:\bbR^d\to\bbR$ is a \emph{potential} for the conservative
field $D$ if for every $x\in \bbR^d$ and every absolutely continuous
 $a \colon [0, 1] \to \bbR^d$, with $a(0) = 0$ and $a(1)=x$,
\begin{equation}
  F(x) = F(0) +  \int_{0}^1 \min_{v \in D(a(t))} \scalarp{\dot a(t)}{v} \dif t \, .
\label{eq:conservative}
\end{equation}
In this case, such a function $F$
is locally Lipschitz continuous,
and for every absolutely continuous curve
$a:[0,1]\to \bbR^d$, the function $t\mapsto F(a(t))$ satisfies
for almost every $t\in [0,1]$,
$$
\frac{d}{dt}F(a(t)) = \ps{v,\dot a(t)}\qquad(\forall v\in D(a(t)))\,,
$$
that is to say, $F$ admits a ``chain rule'' \cite[Lemma 2]{bolte2019conservative}.
Moreover, by \cite[Theorem 1]{bolte2019conservative},
it holds that $D=\{\nabla F\}$ Lebesgue almost everywhere.

We say that a function $F$
is \emph{path differentiable} if there exists a conservative field $D$
such that $F$ is a potential for $D$.
If $F$ is path differentiable, then the Clarke subdifferential $\partial F$
is a conservative field for the potential $F$
\cite[Corollary 2]{bolte2019conservative}.
Another useful example of a conservative field for composite functions is
the automatic differentiation field \cite[Section 5]{bolte2019conservative}.
A broad class of functions used in optimization are path differentiable,
e.g. any convex, concave, regular or tame.
A tame function is a function defined in some o-minimal structure
(\cite{van96}), they enjoy some nice stability properties such as any
elementary operation on them remain tame (e.g. composition, sum,
inverse). The domain $f$ of a tame function admits a so-called Whitney
stratification, that is to say a collection of manifolds $(S_i)$ on
each of which $f$ is smooth with the additional property that the various
gradients fit well together (see \cite{bolte2007clarke} for more
details). The exponential and the logarithm are tame, as well as any
semialgebraic function, an interested reader can find more on tameness
and its usefulness in optimization in \cite{iof08}, and more details
in \cite{van96}, \cite{bolte2007clarke} and \cite{dav-dru-kak-lee-19}.

A similar point of view on differentiation of non-smooth functions is given by
the generalized subdifferential introduced by Norkin \cite{nor_gen80}. A
function $F: \RR^d \rightarrow \RR$ is said to be differentiable in a
generalized sense if there is a set-valued map $G_{F}: \RR^d \rightrightarrows
\RR^d$ such that for every $x$, $G_F(x)$ is nonempty, convex, compact valued,
the graph of $G_F$ is closed, and
\begin{equation*}
  F(y) = F(x) + \scalarp{g(y)}{y-x} + o(x, y, g)\, , \quad{} \textrm{with } g(y) \in G_{F}(y) \textrm{ and } \lim_{y \rightarrow x}\sup_{g \in G_{F}(y)} \frac{o(x,y,g)}{\norm{x - y}} = 0 \, .
\end{equation*}
As in the path-differentiable case, the class of such functions contains tame,
regular and Whitney stratifiable functions. A nice feature of this class is
that, under mild conditions, it is closed with respect to the expectation. That
is to say, if $f: \bbR^d \times \Xi \rightarrow \bbR$ is such that for every $s
\in \Xi$, $f(\cdot, s)$ differentiable in a generalized sense, then the same is
true for $F(x) := \int f(x, s) \mu(\dif s)$ \cite{mikh_gup_nor87}.
Stochastic algorithms with decreasing steps involving the generalized
subdifferential were studied in \cite{erm-nor-98,rusz_20}.

\subsection{Differential Inclusions}\label{subsec:dif_incl}

We endow the set of continuous function from $\bbR_{+}$ to $\bbR^d$ with the metric of uniform convergence on compact intervals of $\bbR_{+}$:
\begin{equation}\label{eq:dC}
\bs d_C (\sx,\sy) = \sum_{n \in \bbN}2^{-n}
  \left(1\wedge \sup_{t\in [0,n]}\|\sx(t)-\sy(t)\|\right)
\end{equation}
Given a set valued map $\sH:\bbR^d \rightrightarrows \bbR^d$, we
say that $x:\bbR_{+} \rightrightarrows \bbR^d$ is a solution of the differential inclusion
\begin{equation}
  \dot{x}(t) \in \sH(x(t))
\label{eq:dif_incl}
\end{equation}
with initial condition $x_0\in \bbR^d$, if $x$ is absolutely continuous,
$x(0) = x_0$ and \eqref{eq:dif_incl} holds for almost every  $t\in\bbR_{+}$.
We denote by $\cS_{\sH}:\RR^d\rightrightarrows C(\RR_+, \RR^d)$ the set-valued
mapping such that for every $a\in \RR^d$, $\cS_{\sH}(a)$ is set of solutions
of~\eqref{eq:dif_incl} with $x_0 = a$. For every
subset $A\subset E$, we define $\cS_\sH(A) = \bigcup_{a\in A}\cS_\sH(a)$.

If a map $\sH$ has nonempty values we will say that it is upper semicontinuous
if the graph of $\sH$, $\{ (x,y) :  y \in \sH(x) \}$, is closed.
In the case where $\sH$ is upper semicontinuous with compact and convex
values and satisfies the condition
\begin{equation}\label{eq:di_lin_growth}
  \exists K  \geq 0, \ \forall x \in \bbR^d,\   \sup \{ \| v \| : v\in H(x)\}\leq K(1+\|x\|)
\end{equation}
then $\cS_\sH(a)$ is non empty for each $a \in \RR^d$, and moreover,
$\cS_\sH(\RR^d)$ is closed in the metric space $(C(\RR_+, \RR^d), \bs d_C)$
\cite{aub-cel-(livre)84}.
The Clarke subdifferential of a locally Lipschitz function
is upper semicontinuous set valued map with
nonempty compact convex values \cite[Chap.~3]{cla-led-ste-wol-livre98}.

\section{Almost-Everywhere Gradient Functions}
\label{sec:ae-gradient}

\subsection{Definition}

Let $(\Xi, \mcT, \mu)$ be a probability space, where the $\sigma$-field $\mcT$
is $\mu$-complete. Let $d>0$ be an integer.
Consider a function $f : \bbR^d \times \Xi \to \RR$. We denote by
$\Delta_f \eqdef \{ (x,s) \in \RR^d \times \Xi \, : \, x\in \cD_{f(\cdot,
s)} \}$ the set of points $(x,s)$ s.t. $f(\cdot,s)$ is differentiable at
$x$. We denote by $\nabla f(x,s)$ the gradient of $f(\cdot,s)$ at $x$,
whenever it exists.

The following technical lemma, the proof of which is provided in
Section~\ref{sec:proofnabmes}, is essential.
\begin{lemma}
\label{nabmes}
Assume that  $f$ is $\mcB(\RR^d) \otimes
\mcT$-measurable  and that $f(\cdot, s)$ is continuous for every $s\in \Xi$.
Then $\Delta_f\in \mcB(\RR^d) \otimes \mcT$, and the function
$\varphi_0 : \RR^d \times \Xi \to \RR^d$ defined as
\begin{equation}
\label{lazy-phi}
\varphi_0(x,s) = \left\{\begin{array}{cl} \nabla f(x,s) & \text{if }
     (x,s) \in \Delta_f \\
     0 & \text{otherwise,} \end{array}\right.
\end{equation}
is $\mcB(\RR^d) \otimes \mcT$-measurable. Moreover, if $f(\cdot,s)$ is locally
Lipschitz continuous for every $s\in \Xi$, then
$(\lambda\otimes\mu)(\Delta_f^c) = 0$.
\end{lemma}
Thanks to this lemma, the following definition makes sense.
\begin{definition}
\label{def:gradient}
Assume that $f(\cdot,s)$ is locally Lipschitz continuous for every $s\in \Xi$.
A function $\varphi : \RR^d \times \Xi \to \RR^d$ is called an
\emph{almost everywhere (a.e.)-gradient} of $f$ if
 $\varphi = \nabla f$ $\lambda\otimes\mu$-almost everywhere.
\end{definition}
By Lemma~\ref{nabmes}, we observe that a.e.-gradients exist, since
$(\lambda\otimes\mu)(\Delta_f^c) = 0$.
Note that in Definition~\ref{def:gradient}, we do not assume that
$\varphi$ is $\mcB(\RR^d)\otimes\mcT / \mcB(\RR^d)$-measurable. The
reason is that this property is not always easy to check on practical examples.
However, if one denotes by $\overline{\mcB(\RR^d)\otimes\mcT}$ the $\lambda\otimes\mu$ completion of the
$\sigma$-field $\mcB(\RR^d)\otimes\mcT$, an immediate consequence of Lemma~\ref{nabmes}
is that any a.e.-gradient of $f$ is a $\overline{\mcB(\RR^d)\otimes\mcT} /
\mcB(\RR^d)$-measurable function. 

\subsection{Examples}

\noindent{\bf Lazy gradient function.} The function $\varphi_0$ given by Equation~\eqref{lazy-phi}
is an a.e. gradient function.

\medskip

\noindent {\bf Clarke gradient function.}
We shall refer to as a Clarke gradient function as any function $\varphi(x,s)$
such that
\begin{equation}
\begin{cases}
\varphi(x,s) &= \nabla f(x, s)\text{  if }(x,s)\in \Delta_f ,  \\
\varphi(x,s) &\in \partial f(x,s)\text{  otherwise.}
\end{cases}
\label{eq:clarkeSGD}
\end{equation}
Note that the inclusion $\varphi(x,s) \in \partial f(x,s)$ obviously holds
for \emph{all} $(x,s) \in \RR^d \times \Xi$, because
$\nabla f(x,s)$ is an element of $\partial f(x,s)$ when the former exists.
However, conversely, a function $\psi(x,s) \in \partial f(x,s)$ does not
necessarily satisfy $\psi(x,s) = \nabla f(x, s)$ if $(x,s)\in \Delta_f$
(see the footnote\footnote{If a locally Lipschitz function $g$ is differentiable
  at a point $x$, we have $\{\nabla g(x)\} \subset \partial g(x)$ but the
  inclusion could be strict (the two sets are equal if $g$ is
    regular at $x$): for example, $g(x) = x^2 \sin(1/x)$ is s.t. $\nabla g(0) = 0$ and $\partial g(0) =
  [-1, 1]$. There even exist functions for which the set of $x$ s.t. $\{
  \nabla g(x) \} \subsetneq \partial g(x)$ is a set of full measure (see \cite[Proposition 1.9]{leb-79}).
}). By construction, a Clarke gradient function is an a.e. gradient function.

\medskip

\noindent {\bf Selections of conservative fields.}
\begin{proposition}\label{lm:cons_sgd}
Assume that for every $s\in \Xi$,
$f(\cdot, s)$ is locally Lipschitz, path differentiable, and is a potential of some conservative field
$D_s:\bbR^d\rightrightarrows \bbR^d$.
Consider a function $\varphi:\bbR^d\times \Xi\to\bbR^d$ which is
$\mcB(\RR^d)\otimes\mcT / \mcB(\RR^d)$ measurable and satisfies
$\varphi(x,s) \in D_{s}(x)$ for all $(x,s)\in \RR^d\times\Xi$.
Then, $\varphi$ is an a.e. gradient function for $f$.
\end{proposition}

\begin{proof}
Define $A \eqdef \{(x,s) \text{ s.t. } \varphi(x,s) \neq \nabla f(x,s)\}$. Applying Fubini's theorem we have:
\begin{equation*}
  \int 1_A(z) \lambda \otimes \mu(\dif z) = \int \int 1_A((x,s)) \lambda (\dif x) \mu (\dif s) = 0 \, ,
\end{equation*}
where the last equality comes from the fact that for every $s$, $D_s=\{ \nabla f(\cdot, s)\}$ $\lambda$-a.e. \cite[Theorem 1]{bolte2019conservative}.
\end{proof}
We provide below an application of Proposition~\ref{lm:cons_sgd}.

\medskip

\noindent {\bf Autograd function.} Consider Equation~\eqref{cnn}, which represents a loss of a neural network. Although $f$ is just a composition of some simple functions, a direct calculation of the gradient (if it exists) may be tedious. Automatic differentiation deals with such functions by recursively applying the chain rule to the components of $f$.
More formally consider a function $f : \bbR^d \to \bbR$ that can be written as a closed formula of simple functions, mathematically speaking this means that we can represent $f$ by a directed graph. This graph (with $q>d$ vertices) is defined through a set-valued function $\bold{parents}(i) \subset \{1, \dots, i-1\}$, a directed edge in this setting will be $j \rightarrow i$ with $j \in \bold{parents}(i)$. Associate to each vertex a simple function
$g_i \colon \bbR^{|\bold{parents}(i)|} \to \bbR$, given an input $ x = (x_1, \dots, x_d) \in \bbR^d$ we recursively define  $x_i = g_i((x_{j})_{j \in \bold{parents}(i)})$ for $i > d$ and finally $f(x) = x_q$. For instance, if $f$ is a cross entropy loss of a neural network, with activation functions being  ReLu or sigmoid functions, then $g_i$ are some compositions of simple functions $\bold{log}$, $\bold{exp}$, $\bold{\frac{1}{1 + x^2}}$,
 norms and piecewise polynomial functions, all being path differentiable \cite[section 6]{bolte2019conservative}, \cite[Section 5.2]{dav-dru-kak-lee-19}.
Automatic differentiation libraries calculate the gradient of $f$ by successively applying the chain rule (in the sense $(g_1 \circ g_2)' = (g_1'\circ g_2 )g_2'$) to the simple functions $g_i$. While the chain rule is no longer valid in a nonsmooth setting (see e.g. \cite{kak_lee19}), it is shown in \cite[Section 5]{bolte2019conservative} that when the simple functions are path-differentiable, the output of automatic differentiation (e.g. {\tt{autograd}} in PyTorch (\cite{pytorch})) is a selection of some conservative field $D$ for $f$.
We refer to \cite{bolte2019conservative} for a more detailed account.
We denote by $a_f(x)$ the output of automatic differentiation of a function $f$ at some point $x$.

Assume that $\Xi = \bbN$ and for each $s \in \Xi$, $f(\cdot, s)$ is defined through a recursive graph of path differentiable functions (in the machine learning paradigm $f(\cdot, s)$ will represent the loss related to one data point, while $F(\cdot)$ is the average loss).
By Proposition~\ref{lm:cons_sgd}, the map $(x,s)\mapsto a_{f(\cdot, s)}(x)$ is an a.e. gradient function for $f$.

\medskip

\noindent{\bf Selections of generalized subdifferentials of Norkin.}
Noticing that a generalized subdifferential of a function is equal to its gradient a.e. (\cite[Theorem 1.12]{mikh_gup_nor87}), the proof of the next proposition is identical to the one of Proposition~\ref{lm:cons_sgd}.
\begin{proposition}
Assume that for every $s\in \Xi$,
$f(\cdot, s)$ is differentiable in a generalized sense, with $G_{f(\cdot, s)} : \bbR^d \rightrightarrows\bbR^d $ being its generalized subdifferential.
Consider a function $\varphi:\bbR^d\times \Xi\to\bbR^d$ which is
$\mcB(\RR^d)\otimes\mcT / \mcB(\RR^d)$ measurable and satisfies
$\varphi(x,s) \in G_{f(\cdot,s)}(x)$ for all $(x,s)\in \RR^d\times\Xi$.
Then, $\varphi$ is an a.e. gradient function for $f$.
\end{proposition}

\section{SGD Sequences}
\label{sec:sgd}
\subsection{Definition}

Given a probability measure $\nu$ on $\mcB(\RR^d)$, define the
probability space $(\Omega, \mcF, \PP^\nu)$ as $\Omega = \RR^d \times
\Xi^\NN$, $\mcF = \mcB(\RR^d) \otimes \mcT^{\otimes \NN}$, and
$\PP^\nu = \nu \otimes \mu^{\otimes \NN}$. We denote by $(x_0,(\xi_n)_{n\in \NN^*})$
the canonical process on $\Omega\to\RR^d$ \emph{i.e.},
writing an elementary event in the space $\Omega$ as $\omega=(\omega_n)_{n\in \NN}$, we
set $x_0(\omega) = \omega_0$ and
$\xi_n(\omega)=\omega_n$ for each $n\geq 1$.
Under $\PP^\nu$, $x_0$ is a $\RR^d$-valued random variable with the probability
distribution $\nu$, and the process $(\xi_n)_{n\in\NN^*}$ is an independent and
identically distributed (i.i.d.) process such that the distribution of
$\xi_1$ is $\mu$, and $x_0$ and $(\xi_n)$ are independent.
We denote by $\overline\mcF$ the
$\lambda\otimes\mu^{\otimes\NN}$-completion of $\mcF$.

Let $f : \bbR^d \times \Xi \to \RR$ be a $\mcB(\RR^d)\otimes \mcT / \mcB(\RR)$-measurable function.
\begin{definition}
\label{def:sgd}
Assume that $f(\cdot,s)$ is locally Lipschitz continuous for every $s\in\Xi$.
A sequence $(x_n)_{n\in\NN^*}$ of functions on $\Omega\to \bbR^d$
is called an SGD sequence for $f$ with the
step $\gamma > 0$ if there exists an a.e.-gradient $\varphi$ of $f$ such that
\[
 x_{n+1} = x_n - \gamma\varphi(x_n, \xi_{n+1}) \qquad (\forall n \geq 0)\,.
\]
\end{definition}

\subsection{All SGD Sequences Are Almost Surely Equal}

Consider the SGD sequence
\begin{equation}
\label{eq:simpleSGD}
x_{n+1} = x_n - \gamma\varphi_0(x_n,\xi_{n+1}),
\end{equation}
generated by the lazy a.e. gradient $\varphi_0$. Denote by
$\Pg : \RR^d \times \mcB(\RR^d) \to [0,1]$ the kernel
of the homogeneous Markov process defined by this equation, which exists thanks
to the $\mcB(\RR^d)\otimes\mcT$-measurability of $\varphi_0$. This kernel is
defined by the fact that its action on a measurable function
$g : \RR^d \to \RR_+$, denoted as $\Pg g(\cdot)$, is
\begin{equation}
\label{kernel}
\Pg g(x) = \int g(x - \gamma\varphi_0(x,s)) \, \mu(ds) .
\end{equation}
Define $\Gamma$ as the set of all steps $\gamma>0$ such that $P_\gamma$ maps $\cM_{abs}(\RR^d)$ into itself:
$$
\Gamma \eqdef \{\gamma\in (0,+\infty)\,:\, \forall \rho\in
\cM_{abs}(\RR^d),\ \rho\Pg\ll\lambda\}\,.
$$
\begin{proposition}
\label{tout-se-vaut}
Consider $\gamma\in \Gamma$ and $\nu\in\cM_{abs}(\RR^d)$.
Then, each SGD sequence $(x_n)$ with the step $\gamma$ is $\overline\mcF /
\mcB(\RR^d)^{\otimes\NN}$-measurable. Moreover, for any two SGD sequences
$(x_n)$ and $(x'_n)$ with the step $\gamma$, it holds that $\PP^\nu \left[ (x_n) \neq (x'_n) \right] =
0$. Finally, the probability distribution of $x_n$ under $\bbP^\nu$ is Lebesgue-absolutely
continuous for each $n\in\NN$.
\end{proposition}
Note that $\PP^\nu \ll \lambda \otimes \mu^{\otimes\NN}$ since $\nu\ll\lambda$.
Thus, the probability $\PP^\nu \left[ (x_n) \neq (x'_n) \right]$ is
well-defined as an integral w.r.t.~$\lambda \otimes \mu^{\otimes\NN}$.
\begin{proof}
Let $(x_n)$ be the lazy SGD sequence given by (\ref{eq:simpleSGD}).
Given an a.e. gradient
$\varphi$, define the SGD sequence $(z_n)$ as $z_0 = x_0$, $z_{n+1} = z_n -
\gamma\varphi(z_n, \xi_{n+1})$ for $n\geq 0$.  The sequence $(x_n)$
is ${\mcF}/ \mcB(\RR^d)^{\otimes\NN}$-measurable thanks to Lemma~\ref{nabmes}.
Moreover, applying recursively the property that $\rho\Pg\ll\lambda$ when
$\rho\ll\lambda$, we obtain that the distribution of $x_n$ is absolutely
continuous for each $n\in\NN$.

To establish the proposition, it suffices to show that $z_n$ is $\overline\mcF
/ \mcB(\RR^d)$-measurable for each $n\in\NN$, and that $\PP^\nu[ z_n \neq x_n ]
= 0$, which results in particular in the absolute continuity of the
distribution of $z_n$. We shall prove these two properties by induction on $n$.
They are trivial for $n = 0$. Assume they are true for $n$.  Recall that
$z_{n+1} = z_n - \gamma \nabla f(z_n, \xi_{n+1})$ if $(z_n, \xi_{n+1}) \in A$,
where $A\in \overline{\mcB(\RR^d)\otimes\mcT}$ is such that
$(\lambda\otimes\mu)(A^c) = 0$, and $x_{n+1} = x_n - \gamma \nabla f(x_n,
\xi_{n+1}) \1_{(x_n, \xi_{n+1}) \in\Delta_f}$.  The set $B = \{ \omega\in \Omega \,
: \, z_{n+1} \neq x_{n+1} \}$ satisfies $B \subset B_1 \cup B_2$, where
\[
B_1 = \{ \omega\in \Omega \, : \, z_n \neq x_n \} \quad \text{and} \quad
       B_2 = \{ \omega\in \Omega \, : \, (z_n,\xi_{n+1}) \not\in A \}.
\]
By induction, $B_1 \in \overline\mcF$ and
  $\PP^\nu(B_1) = 0$. By the aforementioned properties of $A$, the
$\overline{\mcF}$-measurability of $z_n$, and the absolute continuity of its
distribution, we also obtain that $B_2 \in \overline\mcF$ and $\PP^\nu(B_2) = 0$.
Thus, $B \in \overline\mcF$ and $\PP^\nu(B) = 0$, and since $x_{n+1}$ is
$\mcF$-measurable, $z_{n+1}$ is $\overline{\mcF}$-measurable.
\end{proof}
Proposition~\ref{tout-se-vaut} means that the SGD sequence does not depend on
the specific a.e. gradient used by the practioner, provided that the law of
$x_0$ has a density and $\gamma\in \Gamma$. Let us make this last assumption
clearer.  Consider for instance $d=1$ and suppose that $f(x,s) = 0.5 x^2$ for
all $s$.  If $\gamma = 1$, the SGD sequence $x_{n+1} = x_n -\gamma x_n$
satisfies $x_1=0$ for any initial point and thus, does not admit a density,
whereas for any other value of $\gamma$, $x_{n}$ has a density for all $n$,
provided that $x_0$ has a density. Otherwise stated,
$\Gamma = \RR_+ \setminus \{ 1 \}$ in this example.

It is desirable to ensure that $\Gamma$ contains almost all the points of
$\RR_+$. The next proposition shows that this will be the case under mild
conditions. The proof is given in~\ref{sec:proofC2}.
\begin{proposition}
\label{prop:C2}
Assume that for $\mu$--almost every $s\in \Xi$, the function $f(\cdot,s)$
satisfies the property that at $\lambda$--almost every point of $\RR^d$,
there is a neighborhood of this point on which it is $C^2$.
Then, $\Gamma^c$ is Lebesgue negligible.
\end{proposition}
This assumption holds true as soon as for $\mu$-almost all $s$, $f(\cdot,s)$ is tame, since in this case $\bbR^d$ can be partitioned in manifolds on each of which $f(\cdot, s)$ is $C^2$ (\cite{bolte2007clarke}), and therefore $f(\cdot, s)$ is $C^2$ (in the classical sense) on the union of manifolds of full dimension, and therefore almost everywhere.


\subsection{SGD as a Robbins-Monro Algorithm}

We make the following assumption on the function $f:\RR^d\times \Xi\to\RR$.
\begin{assumption}\label{hyp:model}
  \begin{enumerate}[i)]
\item\label{floclip}
  There exists a measurable function $\kappa \, : \, \bbR^d \times \Xi \to \bbR_{+}$
  s.t. for each $x \in \bbR^d$, $\int \kappa(x,s) \, \mu(ds) < \infty$ and there exists $\varepsilon > 0$ for which
    \[
    \forall y, z \in B(x,\varepsilon), \
    \forall s \in \Xi, \
     | f(y,s) - f(z,s) | \leq \kappa(x,s) \| y - z \| .
    \]
\item There exists $x \in \bbR^d$ such that $f(x,\cdot)$ is $\mu$-integrable.
  \end{enumerate}
\end{assumption}
By this assumption, $f(x,\cdot)$ is $\mu$-integrable for each $x \in \RR^d$, and the function
\begin{equation}
F : \RR^d \to \RR, \quad x \mapsto \int f(x,s) \, \mu(ds)
\label{eq:F}
\end{equation}
is locally Lipschitz on $\RR^d$. We denote by $\cZ$ the set
of (Clarke) critical points of $F$, as defined in Equation~(\ref{eq:S}).

Let $(\mcF_n)_{n\geq 0}$ be the filtration $\mcF_n =
\sigma(x_0,\xi_1,\dots,\xi_n)$. We denote by $\bbE_{n} =
\bbE[ \cdot | \overline{\mcF}_n]$ the conditional expectation
w.r.t. $\overline{\mcF}_n$, where $\overline{\mcF}_n$, stands for the
$\lambda \otimes \mu^{\bbN}$-completion of ${\mcF}_n$.
\begin{theorem}
  \label{th:rm}
Let Assumption~\ref{hyp:model} holds true.
Consider $\gamma\in \Gamma$ and $\nu \in \cM_{abs}(\bbR^d) \cap \cM_{1}(\bbR^d)$.
Let $(x_n)_{n\in \NN^*}$ be a SGD sequence for $f$ with the step $\gamma$.
Then, for every $n\in \NN$, it holds $\PP^\nu$-a.e. that
 \begin{enumerate}[i)]
   \item $F$,  $f(\cdot, \xi_{n+1})$ and $f(\cdot, s)$ (for $\mu$-almost every $s$) are differentiable at $x_n$.
   \item $x_{n+1} = x_n - \gamma \nabla f(x_n, \xi_{n+1})$.
   \item $\bbE_n[x_{n+1}] = x_n - \gamma \nabla F(x_n)$.
 \end{enumerate}
\end{theorem}

Theorem~\ref{th:rm} is important because it shows that $\PP^\nu$-a.e., the SGD
sequence $(x_n)$ verifies
$$
x_{n+1} =x_n - \gamma \nabla F(x_n) + \gamma\eta_{n+1}
$$
for some random sequence $(\eta_n)$ which is a martingale difference sequence
adapted to $(\overline\mcF_n)$.

\section{Dynamical Behavior}\label{sec:dyn_beh}

\subsection{Assumptions and Result}


In this section we prove that the SGD sequence $(x_n)_{n\in \NN^*}$ (which is by
Theorem~\ref{th:rm}, under the stated
assumptions, unique) closely follows a trajectory of a solution to the DI~(\ref{eq:DI})
as the step size $\gamma$ tends to zero.
To state the main result of this section, we need to strengthen
Assumption~\ref{hyp:model}.
\begin{assumption}
\label{hyp:model-reinf}
The function $\kappa$ of Assumption~\ref{hyp:model} satisfies:
\begin{enumerate}[i)]
  \item There exists a constant $K \geq 0$ s.t. $\int \kappa(x,s) \, \mu(ds) \leq K ( 1 + \| x \|)$ for all $x$.
  \item For each compact set $\mK \subset \RR^d$, $\sup_{x\in \mK} \int \kappa(x,s)^{2} \mu(ds) < \infty$.
\end{enumerate}
\end{assumption}
The first point guarantees the existence of global solutions
to~\eqref{eq:DI} starting from any initial point (see Section~\ref{subsec:dif_incl}).
\begin{assumption}\label{hyp:zero_in_gamma}
  The closure of $\Gamma$ contains $0$.
\end{assumption}
By Proposition~\ref{prop:C2}, Assumption~\ref{hyp:zero_in_gamma} is mild. It holds for instance
if every $f(\cdot, s)$ is a tame function.

We recall that $\cS_{-\partial F}(A)$
is the set of solutions to (\ref{eq:DI}) that start from any point in the set $A\subset \RR^d$.
\begin{theorem}
\label{th:apt}
Let Assumptions~\ref{hyp:model}--\ref{hyp:zero_in_gamma} hold true.
Let $\{(x_n^\gamma)_{n\in \NN^*}:\gamma\in (0,\gamma_0]\}$
be a collection of SGD sequences of steps $\gamma \in (0, \gamma_0]$.
Denote by $\sx^\gamma$ the piecewise affine interpolated process
$$
\sx^\gamma(t) = x_n^\gamma+(t/\gamma-n)(x_{n+1}^\gamma
  -x_n^\gamma)
\qquad \left(\forall t\in [n\gamma, (n+1)\gamma)\right).
$$
Then,
for every compact set $\mK\subset \RR^d$,
  \[
  \forall \varepsilon>0,\ \lim_{\substack{\gamma\to 0\\\gamma\in
      \Gamma}}\left(\sup_{\nu \in \cM_{abs}(\mK)}\, \PP^{\nu}\left(\bs
      d_C(\sx^\gamma,\cS_{-\partial F}(\mK))>\varepsilon\right)
  \right) = 0\, ,
  \]
where the distance ${\bs d}_C$ is defined in~\eqref{eq:dC}.
Moreover, the family of distributions $\{ \PP^{\nu}(\sx^\gamma)^{-1}:\nu
  \in \cM_{abs}(\mK),0<\gamma<\gamma_0,\gamma\in\Gamma \}$ is
  tight.
\end{theorem}
The proof is given in Section \ref{sec:proofapt}.

Theorem~\ref{th:apt} implies that the interpolated process $\sx^\gamma$ converges in probability as $\gamma\to 0$
to the set of solutions to (\ref{eq:DI}). Moreover, the convergence is uniform w.r.t. to the choice of
the initial distribution $\nu$ in the set of absolutely continuous measures supported by a given compact set.

\subsection{Importance of the Randomization of $x_0$}
\label{sec:discussion}

In this paragraph, we discuss the case where $x_0$ is no longer
random, but set to an arbitrary point in $\RR^d$.  In this case, there
is no longer any guarantee that the iterates $x_n$ only hit the points
where a gradient exist. We focus on the case where $(x_n)$ is a
Clarke-SGD sequence of the form~(\ref{eq:clarkeSGD}), where the
function $\varphi$ is assumed $\mcB(\RR^d)\otimes\mcT/\mcB(\RR^d)$ measurable
for simplicity. By Assumption~\ref{hyp:model},
it is not difficult to see that $\varphi(x,\cdot)$ is $\mu$-integrable for all
$x \in \RR^d$ and, denoting by $\EE(\varphi(x,\cdot))$ the corresponding integral w.r.t. $\mu$,
we can rewrite the iterates under the form:
\[
x_{n+1} = x_n - \gamma \bbE \varphi(x_n, \cdot) + \gamma \eta_{n+1},
\]
where $\eta_{n+1} = \bbE[\varphi(x_n, \cdot)] - \varphi(x_n,
\xi_{n+1})$ is a martingale difference sequence for the
filtration~$(\mcF_n)$.  Obviously, $\EE\varphi(x,\cdot) \in
\EE \partial f(x,\cdot)$.  As said in the introduction, we need
$\EE\varphi(x,\cdot)$ to belong to $\partial F(x)$ in order to make
sure that the algorithm trajectory shadows the DI $\dot\sx(t) \in
- \partial F(\sx(t))$. Unfortunately, the inclusion $\partial F(x)
\subset \EE \partial f(x,\cdot)$ can be strict, which can result in
the fact that the DI $\dot\sx(t) \in - \EE \partial f(\sx(t),\cdot)$
generates spurious trajectories that converge to spurious zeroes. The
following example, which can be easily adapted to an arbitrary
dimension, shows a case where this phenomenon happens.
\begin{example}
\label{ex:dif_incl_neq1}
Take a finite probability space $\Xi = \{ 1, 2\}$ and
$\mu(\{1\}) = \mu(\{2\}) = 1/2$.
Let $f(x, 1) =  2x \1_{x \leq 0} $ and $f(x, 2) = 2x \1_{x \geq 0}$. We
have $F(x) = x$, and therefore $\partial F(0) = \{ 1 \} $, whereas $\partial
f(0, 1) = \partial f(0, 2) = [0, 2]$ and therefore $\int \partial f(0, s)
\mu(\dif s) = [0, 1]$.  We see that $0 \in \EE \partial f(0,\cdot)$ while
$0 \not\in \partial F(0)$. Furthermore, the trajectory defined on $\RR_+$ as
\[
\sx(t) = \left\{ \begin{array}{ll} 1 - t &\text{for} \  t \in [0,1] \\
                                0    &\text{for} \ t > 1
         \end{array}\right. , \quad \sx(0) = 1,
\]
is a solution to the DI $\dot\sx(t) \in - \EE \partial f(\sx(t),\cdot)$, but
not to the DI $\dot\sx(t) \in - \partial F(\sx(t))$.
\end{example}
\begin{example}
  Consider the same setting as in the previous example. Consider a stochastic gradient algorithm
  of the form~\eqref{eq:sgd}, initialized at $x_0=0$ with $\varphi$ such that
  $\varphi(0, 1) = \varphi(0, 2) = 0$. Then, the iterates $x_n^\gamma$ are identically zero.
  This shows that the stochastic gradient descent may converge to a non critical point of $F$.
  Theorem~\ref{th:apt} may fail unless a random initial point is chosen.
\end{example}

\section{Long Run Convergence}
\label{sec:longrun}

\subsection{Assumptions and Result}

As discussed in the introduction, the SGD sequence $(x_n)$ is not expected to
converge in probability to $\cZ$ when the step is constant. Instead, we shall
establish the convergence~\eqref{eq:main_res}. The ``long run'' convergence
referred to here is understood in this sense.

In all this section, we shall focus on the lazy SGD sequences described by
Equation~\eqref{eq:simpleSGD}. This incurs no loss of generality, since any two SGD
sequences are equal $\PP^\nu$-a.e.~by Proposition~\ref{tout-se-vaut}  as long as
$\nu\ll\lambda$. Our starting point is to see the process $(x_n)$ as a Markov
process which kernel $\Pg$ is defined by Equation~\eqref{kernel}. Our first task is
to establish the ergodicity of this Markov process under the convenient
assumptions. Namely, we show that $\Pg$ has a unique invariant probability
measure $\pi_\gamma$, \emph{i.e.}, $\pi_\gamma P_{\gamma} = \pi_\gamma$, and that $\|
P_\gamma^n(x,\cdot) - \pi_\gamma \|_{\text{TV}} \to 0$ as $n\to\infty$ for each
$x \in\RR^d$, where $\|\cdot\|_{\text{TV}}$ is the total variation norm. Further, we need to show that the
family of invariant distributions $\{ \pi_\gamma \}_{\gamma\in(0, \gamma_0]}$
for a certain $\gamma_0 > 0$ is tight. The long run behavior referred to above
is then intimately connected with the properties of the accumulation points of
this family as $\gamma \to 0$. To study these properties, we get back to the DI $\dot\sx \in - \partial F(\sx)$ (we recall that a concise account of the notions relative to this
dynamical system and needed in this paper is provided in
Section~\ref{subsec:dif_incl}).  The crucial point here is to show, with the help of
Theorem~\ref{th:apt}, that the accumulation points of $\{ \pi_\gamma \}$ as
$\gamma\to 0$ are invariant measures for the set-valued flow induced by the DI. In its original form, this idea dates back to the work of
Has'minski\u{\i} \cite{has-63}.  We observe here that while the notion of
invariant measure for a single-valued semiflow induced by, say, an ordinary
differential equation, is classical, it is probably less known in the case of
a set-valued differential inclusion. We borrow it
from the work of Roth and Sandholm \cite{rot-san-13}.

Having shown that the accumulation points of $\{ \pi_\gamma \}$ are invariant
for the DI $ \dot{\sx} \in -\partial F(\sx)$, the final step of the proof is to make use of
Poincar\'e's recurrence theorem, that asserts that the invariant measures of a
semiflow are supported by the so-called Birkhoff center of this semiflow
(again, a set-valued version of Poincar\'e's recurrence theorem is provided in
\cite{aub-fra-las-91,fau-rot-13}). To establish the
convergence~\eqref{eq:main_res}, it remains to show that the Birkhoff center
of the DI $\dot{\sx} \in -\partial F(\sx)$ coincides with $\zer \partial F$. The natural
assumption that ensures the identity of these two sets will be that $F$
admits a chain rule
\cite{cla-led-ste-wol-livre98,bolte2007clarke,dav-dru-kak-lee-19}.

\medskip

Our assumption regarding the behavior of the Markov kernel $P_{\gamma}$ reads
as follows.

\begin{assumption}
\label{markov}
There exist measurable functions $V:\RR^d \to [0,+\infty)$, $p:\RR^d \to [0,+\infty)$,
$\alpha:(0,+\infty) \to (0,+\infty)$ and a constant $C\geq 0$ s.t. the following holds
for every $\gamma\in \Gamma\cap (0,\gamma_0]$.
\begin{enumerate}[i)]
\item\label{smallset} There exists $R > 0$ and a positive Borel measure
$\rho$ on $\RR^d$ ($R,\rho$ possibly depending on $\gamma$) such that
\[
\forall x \in \cl({B}(0,R)), \ \forall A \in \mcB(\RR^d), \
 P_{\gamma}(x,A) \geq \rho(A) .
\]
\item\label{drift} $\sup_{\cl({ B}(0,R))}V<\infty$ and $\inf_{B(0,R)^c}p>0$.
Moreover, for every $x\in \RR^d$,
\begin{equation}
\label{eq:drift}
P_{\gamma} V(x) \leq V(x) - \alpha(\gamma) p(x) +
                      C\alpha(\gamma) \1_{\| x \| \leq R} .
\end{equation}
\item\label{tension}
The function $p(x)$ diverges to infinity as $\| x \| \to\infty$.
\end{enumerate}
\end{assumption}
Assumptions of this type are frequently encountered in the field of Markov
chains. Assumption \ref{markov}--\eqref{smallset} states that
$\cl({B}(0,R))$ is a so-called small set for the kernel $\Pg$, and
Assumption \ref{markov}--\eqref{drift} is a standard drift assumption. Taken
together, they ensure that the kernel $\Pg$ is a so-called Harris-recurrent
kernel, that it admits
a unique invariant probability distribution $\pi_\gamma$, and finally, that
this kernel is ergodic in the sense that $\| P_{\gamma}(x,\cdot) - \pi_\gamma
\|_{\text{TV}} \to 0$ as $n\to\infty$ (see \cite{mey-twe-livre09}).
The introduction of the factors $\alpha(\gamma)$ and $C\alpha(\gamma)$
in Equation~(\ref{eq:drift}) guarantees moreover the tightness
of the family $\{ \pi_\gamma \}_{\gamma\in (0, \gamma_0]}$.

In Section~\ref{subsec-suffis}, we provide sufficient and verifiable
conditions ensuring the validity of Assumption~\ref{markov} for $P_{\gamma}$.

As announced above, we also need:
\begin{assumption}
\label{hyp:chrule}
The function $F$ defined by~\eqref{eq:F} admits a chain rule, namely,
for any absolutely continuous curve $z:\bbR_+\to\bbR^d$,
for almost all $t>0$,
$\forall v\in \partial F(z(t)),\  \scalarp{v}{\dot{z}(t)} =  (F \circ z)'(t)\,.$
\end{assumption}
Assumption~\ref{hyp:chrule} is satisfied as soon as $F$ is path-differentiable,
for instance when $F$ is either convex, regular,
Whitney stratifiable or tame (see~\cite[Proposition 1]{bolte2019conservative} and ~\cite{bolte2007clarke,dav-dru-kak-lee-19}).

Since Assumption~\ref{hyp:zero_in_gamma} is satisfied as soon as $f(\cdot, s)$
is tame for every $s \in \Xi$, one can wonder if it can be somehow coordinated
with Assumption~\ref{hyp:chrule}. Unfortunately, $F$ is not necessarily tame
even if $f(\cdot, s)$ is tame for every $s \in \Xi$.
Nonetheless, one can hope that the practical situations where $f(\cdot, s)$ is
tame and $F$ is not are rare. In particular, $F$ will be tame as soon as $\Xi$
is finite (hence the expectation is just a finite sum), which is the case in
many machine learning models.

\begin{theorem}\label{th:main}
Let Assumptions~\ref{hyp:model}-\ref{hyp:zero_in_gamma} and \ref{markov}-\ref{hyp:chrule} hold true.
Let $\{(x_n^\gamma)_{n\in \NN^*}:\gamma\in (0,\gamma_0]\}$
be a collection of SGD sequences of step-size $\gamma$.
Then, the set $ \cZ = \{ x \colon 0 \in \partial F(x) \}$ is nonempty and for all $\nu\in \cM_{abs}(\RR^d)$ and all $\varepsilon>0$,
\begin{equation}
\label{cvg}
\limsup_{n\to\infty}
 \PP^\nu\left( \bs d( x_n^\gamma, \cZ) > \varepsilon \right)
 \xrightarrow[\underset{\gamma \in \Gamma}{\gamma\to 0}]{} 0 .
\end{equation}

%

\end{theorem}

\subsection{On Assumption~\ref{markov}}
\label{subsec-suffis}

In this paragraph, we provide sufficient conditions under which Assumption~
\ref{markov} hold true.
A simple way to ensure the truth of
Assumption~\ref{markov}-\eqref{smallset} is to add a small random perturbation
to the function $\varphi_0(x,s)$.  Formally, we modify algorithms described
by Equation~\eqref{eq:simpleSGD} and~\eqref{eq:sgd_proj}, and write
\begin{equation*}
  \begin{split}
    x_{n+1} &= x_n - \gamma \varphi_0(x_n,\xi_{n+1}) + \gamma\epsilon_{n+1}\\
  \end{split}
\end{equation*}
where $(\epsilon_n)$ is a sequence of centered i.i.d.~random variables of law $\mu^d$, independent from
$\{x_0, (\xi_n)\}$, and such that the distribution of $\epsilon_1 \sim \mu^d$ has a continuous
and positive density on $\RR^d$.   The Gaussian case
$\epsilon_1 \sim {\mathcal N}(0, a I_d)$ where $a > 0$ is some small variance is of
course a typical example of such a perturbation.

 Consider now a fixed $\gamma$ and denote by
$\widetilde P$ the Markov kernel induced by the modified equation.

\begin{proposition}
\label{additive-ss}
Let Assumption~\ref{hyp:model-reinf} hold true. Then, for each $R > 0$,
there exists $\varepsilon > 0$ such that
\[
\forall x \in \cl({B}(0,R)), \ \forall A \in \mcB(\RR^d), \
 \widetilde{P}(x,A)
   \geq \varepsilon \, \lambda (A \cap \cl({B}(0,1))) ,
\]
Thus, Assumption~\ref{markov}-\eqref{smallset} is satisfied for $\widetilde{P}$.
\end{proposition}

We now turn to the assumptions \ref{markov}-\eqref{drift} and
\ref{markov}-\eqref{tension}.

\begin{proposition}
\label{prop:PHsufficient}
Assume that there exists $R \geq 0$, $C>0$, and a measurable function
$\beta:\Xi\to \RR_+$ such that the following conditions hold:
\begin{enumerate}[i)]
\item For every $s \in \Xi$, the function $f(\cdot, s)$ is differentiable outside the ball $\cl(B(0,R))$. Moreover,
for each $x , x' \not\in \cl( B(0,R))$, $\| \nabla f(x,s) - \nabla f(x',s) \| \leq \beta(s) \|x - x'\|$
and $\int \beta^2d\mu<\infty$.
\item For all $x\not\in \cl(B(0,R))$, $\int \| \nabla f(x,s) \|^2 \mu(ds) \leq C ( 1 + \| \nabla F(x) \|^2)$.
\item $\lim_{\|x\|\to\infty} \| \nabla F(x) \|=+\infty$.
\item Function $F$ is lower bounded \emph{i.e.}, $\inf F>-\infty$.
\end{enumerate}
Then, it holds that
\begin{equation}
\label{Fdrift2}
\Pg F(x) \leq F(x)
   - \gamma (1 - \gamma K) \1_{\| x \| > 2R}
       \| \nabla F(x) \|^2 + \gamma^2 K \1_{\| x \| > 2R}
 + \gamma K \1_{\| x \| \leq 2R}
\end{equation}
for some constant $K > 0$.
In particular, Assumptions \ref{markov}-\eqref{drift} and
\ref{markov}-\eqref{tension} hold true.
\end{proposition}

We finally observe that this proposition can be easily adapted to the
case where the kernel $\Pg$ is replaced with the kernel $\widetilde P$ of
Proposition~\ref{additive-ss}.

\section{The Projected Subgradient Algorithm}
\label{sec:proj-sgd}

In many practical settings, the conditions of
Proposition~\ref{prop:PHsufficient} that ensure the truth of
Assumptions~\ref{markov}--\eqref{drift} and~\ref{markov}--\eqref{tension} are
not satisfied. This is for instance the case when the function $f$ is described
by Equation~\eqref{cnn} with the mappings $\sigma_\ell$ at the right hand side
of this equation being all equal to the ReLU function.  In such situations, it
is often pertinent to replace the SGD sequence with a \emph{projected} version
of the algorithm. Given an a.e.-gradient $\varphi$ of the function $f$ and a
non empty compact and convex set $\cK \subset \RR^d$, a \emph{projected SGD
sequence} $(x_n^{\gamma,\cK})$ is given by the recursion
\begin{equation}
\label{eq:sgd_proj}
x_0^{\gamma,\cK}=x_0, \quad
  x_{n+1}^{\gamma,\cK} =
  \Pi_{\cK}(x_n^{\gamma,\cK} - \gamma \varphi(x_n^{\gamma,\cK}, \xi_{n+1}))
   \, ,
\end{equation}
where $\Pi_\cK$ stands for a Euclidean projection onto $\cK$.  Our purpose is
to generalize Theorem~\ref{th:apt} to this situation. This generalization is
not immediate for several reasons.  First, the projection step is likely to
introduce spurious local minima. As far as the iterates~\eqref{eq:sgd_proj} are
concerned, the role of differential inclusion~\eqref{eq:DI} is now played by
the differential inclusion:
\begin{equation}
  \label{eq:projDI}
  \dot \sx(t) \in - \partial F(\sx(t))-\cN_\cK(\sx(t))\,,
\end{equation}
where $\cN_\cK(x)$ stands the normal cone of $\cK$ at point $x$. The set of
equilibria of the above differential inclusion coincides with the set
$$
\cZ_\cK := \{x\in \RR^d\,:\,0\in -\partial F(x) - \cN_\cK(x)\}\,,
$$
which we shall refer to as the set of Karush-Kuhn-Tucker points.  A second
theoretical difficulty is related to the fact that
Proposition~\ref{tout-se-vaut} does no longer hold. Indeed, it can happen $x_0$
has a density, but the next iterates $x_n^{\gamma,\cK}$ don't.  The reason is
that $x_n^{\gamma,\cK}$ generally has a non zero probability to be in the
(Lebesgue negligible) border of $\cK$, that is, $\cl(\cK)\setminus
\inter(\cK)$, where $\cl(\cK)$ and $\inter(\cK)$ respectively stand for the
closure and the interior of $\cK$.

We shall focus here on the case where $\cK = \cl({B}(0,r))$ with $r>0$.  We shall
use $\Pi_r$, $x_n^{\gamma,r}$, $\cN_r$ as shorthand notations for $\Pi_{\cl({B}(0,r))}$,
$x_n^{\gamma,\cl({B}(0,r))}$, and $\cN_{\cl({B}(0,r))}$ respectively. In this case $\cN_r(x) = \{0\}$
if $\norm{x} < r$, $\cN_r(x) = \{ \lambda x: \lambda \geq 0\}$ if $\norm{x} =
r$ and $\cN_r(x) = \emptyset$ otherwise.

We make the following assumption.
\begin{assumption}\label{hyp:proj_ker}
  For every $x \in \bbR^d$, the law of $\varphi_0(x, \xi)$, where $\xi \sim \mu$, is absolutely continuous relatively to Lebesgue.
\end{assumption}

Assumption~\ref{hyp:proj_ker} is much stronger than
Assumption~\ref{hyp:zero_in_gamma}. Indeed, it implies that the distribution of
$x_n^{\gamma,r} - \gamma \varphi(x_n^{\gamma,r}, \xi_{n+1})$ is always
Lebesgue-absolutely continuous. It is useful to note though that
Assumption~\ref{hyp:proj_ker} holds upon adding at each step a small random
perturbation to $\varphi_0$ as in Section~\ref{subsec-suffis} above.

In order to state our first result in this framework, we need to introduce some
new notations.  We let $\SSS(r) := \{ x: \norm{x} = r, x\in \bbR^d \}$ be the
sphere of radius $r$.
By \cite[Theorem 2.49]{folland2013real}, there is a unique measure\footnote{As it is clear from Equation~\eqref{eq:polar_rep} we can see $(\lambda^1,\varrho_1)$ as a polar coordinates representation of the Lebesgue measure $\lambda^d$.} $\varrho_1$ on
$\SSS(1)$ such that for any positive function $f : \bbR^d \rightarrow \bbR$, we
have:
\begin{equation}
\label{eq:polar_rep}
  \int f(x) \lambda^d(\dif x) =
  \int_{0}^{\infty} \int_{\SSS(1)}
     f(r \theta)r^{d-1} \varrho_1(\dif \theta) \lambda^1(\dif r) \, .
\end{equation}
We define the measure $\varrho_r$ on $\SSS(r)$ as $\varrho_r(A) =
\varrho_1(A/r)$ for each Borel set $A \subset \SSS(r)$. We denote as $\cM^r$ the set of measures $\nu = \nu_1 +
\nu_2$, where $\nu_1 \in \cM_{abs}$ and $\nu_2 \ll \varrho_r$. For a set $\mC \subset \bbR^d$ we define $\cM^r(\mC)$ as the measures in $\cM^r$  that are supported on $\mC$. Notice that $\cM_{abs}(\mC) \subset \cM^r(\mC)$.


The next proposition, which is proven in the same way as
Proposition~\ref{tout-se-vaut}, shows that for almost every $r>0$, all
projected SGD sequences are almost surely equal.
\begin{proposition}
  \label{tout-se-vaut-proj}
Let Assumption~\ref{hyp:proj_ker} hold true. Then, for almost every $r>0$, $\forall \nu \in \cM^r$, each projected SGD sequence $(x_n^{\gamma,r})$ is
$\overline\mcF / \mcB(\RR^d)^{\otimes\NN}$-measurable. Moreover, for any two
projected SGD sequences $(x_n^{\gamma,r})$ and $(y_n^{\gamma,r})$, it holds
that $\PP^\nu \left[ (x_n^{\gamma,r}) \neq (y_n^{\gamma,r}) \right] = 0$.
Finally, under $\bbP^{\nu}$, for every $n \in \bbN$, the probability
distribution of $x_n^{\gamma,r}$ is in $\cM^r$.
\end{proposition}

By Proposition~\ref{tout-se-vaut-proj} we can focus on the
lazy projected SGD sequence:
\begin{equation}\label{eq:lazy_proj}
  x_{n+1}^{\gamma,r} = \Pi_r(x_{n}^{\gamma,r} - \gamma \varphi_0(x_n^{\gamma,r}, \xi_{n+1})) \, .
\end{equation}
We define its associated kernel
\begin{equation}\label{kernel-proj}
  \Pg^rg(x) = \int g(\Pi_r(x - \gamma \varphi_0(x, s))) \mu(\dif s)\, .
\end{equation}

The next two theorems are analogous to Theorems~\ref{th:rm} and~\ref{th:apt}.

\begin{theorem}\label{th:rm_proj}
  Let Assumptions~\ref{hyp:model} and \ref{hyp:proj_ker} hold. Then for almost every $r> 0$ , $\forall \nu \in \cM^r$, for every $n \in \bbN$ it holds $\bbP^{\nu}$-a.e.
  \begin{enumerate}[i)]
    \item $F$, $f(\cdot, \xi_{n+1})$ and $f(\cdot, s)$ (for $\mu$-a.e. $s$) are differentiable at $x_n^{\gamma,r}$.
    \item $x_{n+1}^{\gamma,r} \in x_n^{\gamma,r} - \gamma \nabla f(x_n^{\gamma,r}, \xi_{n+1}) - \gamma \cN_r( \Pi_r(x_n^{\gamma,r} - \gamma \nabla f(x_n^{\gamma,r}, \xi_{n+1})))$.
  \end{enumerate}
\end{theorem}

\begin{theorem}
\label{th:apt_proj}
Let Assumptions~\ref{hyp:model}--\ref{hyp:model-reinf} and \ref{hyp:proj_ker} hold true.  Denote $\sx^{\gamma,r}$ the piecewise affine interpolated process:

$$
\sx^{\gamma,r}(t) = x_n^{\gamma,r}+(t/\gamma-n)(x_{n+1}^{\gamma,r}
  -x_n^{\gamma,r})
\qquad \left(\forall t\in [n\gamma, (n+1)\gamma)\right).
$$
Then, for almost every $r>0$, for every compact set  $\mK\subset \cl({B}(0,r))$,
  \[
  \forall \varepsilon>0,\ \lim_{\gamma \rightarrow 0}\left(\sup_{\nu \in \cM^r(\mK)}\, \PP^{\nu}\left(\bs
      d_C(\sx^{\gamma,r},\cS_{- \partial F - \cN_r}(\mK))>\varepsilon\right)
  \right) = 0\, .
  \]
Moreover, for any $\gamma_0 > 0$, the family of distributions $\{ \PP^{\nu}(\sx^{\gamma,r})^{-1}:\nu \in \cM^r(\mK),0<\gamma <\gamma_0\}$ is
  tight.
\end{theorem}

We compare Theorems~\ref{th:rm} and~\ref{th:apt}. First, because of the projection step (and with the help of
Assumption~\ref{hyp:proj_ker}), the law of the $n$-th iterate is no longer in
$\cM_{abs}$, but in $\cM^r$. Second, the continuous counterpart of
Equation~\eqref{eq:sgd_proj} is now the differential inclusion~\eqref{eq:projDI}
%
Note that, if the solutions of the DI~\eqref{eq:DI} that start from $\cK$ all lie
in $\cl({B}(0,r))$, then the set of these solutions coincides with the set
of solutions of the DI~\eqref{eq:projDI} that start from $\cK$.

The analysis of the convergence of the iterates in the "long run"  is greatly simplified by the introduction of the projection step. Compared to Assumption~\ref{markov}, we only assume the existence of a small set for $\Pg^r$, the drift condition of the form \ref{markov}-\eqref{drift}--\eqref{tension} is then automatically satisfied, thanks to the projection step (see Section~\ref{sec:proof_main}).
\begin{assumption}\label{hyp:smallset_proj}
  There is $R >0$ and $\gamma_0>0$ such that for every $\gamma \in (0, \gamma_0]$ there is $\rho_\gamma$ such that Assumption~\ref{markov}-\eqref{smallset} hold for $(R, \rho_\gamma)$ (note that $R$ is independent of $\gamma$ here).
\end{assumption}
As shown in Section~\ref{subsec-suffis}, Assumption~\ref{hyp:smallset_proj} holds upon adding to $\varphi_0$ a small random perturbation.


 \begin{theorem}\label{th:main_proj}
   Let Assumptions~\ref{hyp:model}-\ref{hyp:model-reinf} and \ref{hyp:chrule}--\ref{hyp:smallset_proj} hold. Let $\{(x_n^{\gamma,r})_{n\in \NN^*}:\gamma\in (0,\gamma_0]\}$
   be a collection of projected SGD sequences of step-size $\gamma$.
   Then, for almost every $ 0 < r  \leq R$, the set $\cZ_r = \{ x \colon 0 \in \partial F(x) + \cN_r(x) \}$ is nonempty and for all $\nu\in \cM^r$ and all $\varepsilon>0$,
   \begin{equation}
   \label{cvg-proj}
   \limsup_{n\to\infty}
    \PP^\nu\left( \bs d( x_n^{\gamma,r}, \cZ_r) > \varepsilon \right)
    \xrightarrow[{\gamma\to 0}]{} 0 .
   \end{equation}
 \end{theorem}

Theorem~\ref{th:main_proj} is analogous to Theorem~\ref{th:main}. Notice that, since $\cM_{abs} \subset \cM^r$, $x_0$ can still be initialized under a Lebesgue-absolutely continuous measure. On the other hand, as explained in the beginning of this section, due to the projection step, the iterates, instead of converging to $\cZ$, are now converging to the set of Karush-Kuhn-Tucker points related to the DI~\eqref{eq:projDI}.

\section{Proofs}
\label{sec:proofs}

\subsection{Proof of Lemma~\ref{nabmes}}
\label{sec:proofnabmes}

By definition, $(x,s)\in \Delta_f$ means that there exists $d_x \in \bbR^d$ (the gradient) s.t.
$f(x + h,s) = f(x,s) + \langle d_x, h \rangle + o(\norm{h})$. That is to
say $(x,s)$ belongs to the set:
\begin{equation}\label{eq:dif_condition}
  \bigcap_{\varepsilon \in \bbQ} \bigcup_{\delta \in \bbQ}\bigcap_{ 0< \norm{h}\leq \delta}\left\{ (y,s): \left|\frac{f(y+h, s) - f(y, s) - \langle d_x, h \rangle}{\norm{h}}\right| < \varepsilon \right \}\,.
\end{equation}
In addition, using that  $f( \cdot, s)$ is continuous, the above set is unchanged if the inner intersection
over $0 < \norm{h} \leq \delta$ is replaced by an intersection over the $h$ s.t. $0<\norm{h}\leq \delta$
and having \emph{rational} coordinates \emph{i.e.}, $h\in \bbQ^d$.
Define:
\begin{equation}
  \Delta_f' :=  \bigcap_{\varepsilon' \in \bbQ} \bigcup_{d \in \bbQ^d}\bigcap_{\varepsilon \in \bbQ} \bigcup_{\delta \in \bbQ}\bigcap_{ \underset{ h \in \bbQ^d }{0<\norm{h} \leq \delta} } \left\{ (x, s ) : \left|\frac{f(x+h, s) - f(x, s) - \langle d, h \rangle}{\norm{h}}\right| < \varepsilon + \varepsilon' \right \}
\end{equation}
By construction, $\Delta_f'$ is a measurable set. We prove that $\Delta_f'=\Delta_f$.
Consider $(x,s ) \in \Delta_f$ and let $d_x$ be the gradient of $f(\cdot, s)$ at $x$.
By  \eqref{eq:dif_condition} for all $\varepsilon \in \bbQ$, there is a $\delta \in \bbQ$ such that:
\begin{equation*}
  (x,s ) \in \bigcap_{h \leq \delta, h \in \bbQ^d}\left\{\left|\frac{f(x+h, s) - f(x, s) - \langle d_x, h \rangle}{h}\right| < \varepsilon \right \}
\end{equation*}
For any $\varepsilon' > 0$, choose $d' \in \bbQ^d$ such that $\norm{d' - d_x} \leq \varepsilon'$. Using the previous inclusion,
for all $\varepsilon$, there exists therefore $\delta\in \bbQ$ s.t.
\begin{equation*}
    (x,s ) \in \bigcap_{h \leq \delta, h \in \bbQ^d}\left\{\left|\frac{f(x+h, s) - f(x, s) - \langle d_q, h \rangle}{h}\right| < \varepsilon + \varepsilon' \right \}
\end{equation*}
which means $\Delta_f \subset \Delta'_f$. To show the converse, consider $(x, s) \in \Delta'_f$.
Let $(\varepsilon'_k)$ be a positive sequence of rationals converging to zero. By definition, for every $k$,
there exists $d_k\in \bbQ^d$ s.t. for all $\varepsilon$, there exists $\delta_k(\varepsilon)$, s.t. for all (rational) $h\leq \delta_k(\varepsilon)$,
\begin{equation}
\label{eq:difdk}
\left|\frac{f(x+h, s) - f(x, s) - \langle d_k, h \rangle}{h}\right| < \varepsilon + \varepsilon_k'\,.
\end{equation}
Moreover, one may choose $\delta_k(\varepsilon)\leq \delta_0(\varepsilon)$.
Inspecting first the inequality~(\ref{eq:difdk}) for $k=0$, we easily obtain that the quantity $\frac{f(x+h, s) - f(x, s)}{h}$ is bounded
uniformly in  $h$ s.t. $0<\|h\|\leq \delta_0(\varepsilon)$. Using this observation and again Equation~(\ref{eq:difdk}),
this in turn implies that $({d_k})$ is a  bounded sequence. There exists $d\in \bbR^d$ and s.t. $d_k \rightarrow d$ along some
extracted subsequence.
Now consider $\varepsilon>0$ and choose $k$ such that $\norm{d_k - d} < \frac{\varepsilon}{2}$ and $\varepsilon'_k<\frac\varepsilon 2$.
For all $h \leq \delta_k(\varepsilon/2)$,
\begin{equation*}
  \left|\frac{f(x+h, s) - f(x, s) - \langle d, h \rangle}{h}\right| \leq \left|\frac{f(x+h, s) - f(x, s) - \langle d_k, h \rangle}{h}\right| + \norm{d - d_k} < \varepsilon
\end{equation*}
This means that $d$ is the gradient of $f(\cdot, s)$ at $x$, hence $\Delta'_f \subset \Delta_f$.
Hence, the first point of the Lemma~\ref{nabmes} is proved.

Denoting as $e_i$ the $i^{\text{th}}$ canonical vector of $\RR^d$, the
$i^{\text{th}}$-component $[\varphi_0]_i$ in $\RR^d$ of the function
$\varphi_0$ is given as
\[
[\varphi_0(x,s)]_i = \lim_{t\to 0}
 \frac{f(x+te_i, s) - f(x,s)}{t} \1_{\Delta_f}(x,s),
\]
and the measurability of $\varphi_0$ follows from the measurability of $f$ and
the measurability of $\1_{\Delta_f}$.

Finally, assume that $f(\cdot,s)$ is locally Lipschitz continuous for every
$s\in \Xi$.  From Rademacher's theorem \cite[Ch.~3]{cla-led-ste-wol-livre98},
$f(\cdot, s)$ is almost everywhere differentiable, which reads $\int (1 -
\1_{\Delta_f}(x,s)) \lambda(dx)=0$.  Using Fubini's theorem, $\int_{\RR^d
\times\Xi} (1 - \1_{\Delta_f}(x,s)) \ \lambda(dx) \otimes \mu(ds) = 0$, and the
last point is proved.

\subsection{Proof of Proposition~\ref{prop:C2}}\label{sec:proofC2}

The idea of the proof is to show that for almost every $\gamma$ and $s$ we have that $g_{s, \gamma}(x) := (x - \gamma \nabla f(x, s)) \1_{\Delta_f}(x,s)$ is almost everywhere a local diffeomorphism.

In order to prove that we define for each $(x,s) \in \bbR^d \times \Xi$ the pseudo-hessian  $\cH(x,s) \in \bbR^{d \times d}$ as
\begin{equation*}
 \cH(x,s)_{i, j} = \limsup_{t \rightarrow 0} \frac{\scalarp{\nabla f(x + t e_j, s)\1_{\Delta_f}(x+ te_j,s ) - \nabla f(x,s)}{e_i}}{t} \1_{\Delta_f}(x,s)  \, .
\end{equation*}

Since it is a limit of measurable functions, $\cH$ is $\cB(\bbR^d) \otimes \mcT$ measurable, and if $f(\cdot, s)$ is two times differentiable at $x$ then $\cH(x,s)$ is just the ordinary hessian. Now we define
$l(x, s, \gamma) = \det (\gamma \cH(x,s) - \Id)$ if every entry in $\cH(x,s)$ is finite, and $l(x, s, \gamma) = 1$ otherwise, it is a $\cB(\bbR^d) \otimes \mcT \otimes \cB(\bbR_{+})$ measurable function (as a sum of two measurable functions). By the inverse function theorem we have that if $f(\cdot, s)$ is $C^2$ at $x$ and if $\det (\gamma \cH(x,s) - \Id) \neq 0$, then $g_{s, \gamma}(\cdot)$ is a local diffeomorphism at $x$. Therefore $l(x, s, \gamma) \neq 0$ implies either the latter or $f(\cdot, s)$ is not $C^2$ at $x$ (or both).\\
Let $\lambda^d, \lambda^1$ denote Lebesgue measures respectively on $\bbR^d$ and $\bbR_{+}$, we  have by Fubini's theorem:
\begin{equation*}
\begin{split}
    \int \1_{l(x, s, \gamma) = 0}  \lambda^d(\dif x) \otimes \mu(\dif s) \otimes \lambda^1(\dif \gamma) &= \int \lambda^d\otimes \mu(\{ (x, s): l(x,s, \gamma) = 0\}) \lambda^1(\dif \gamma)\\
    &=\int \int \int \1_{l(x, s, \gamma) = 0}  \lambda^1(\dif \gamma) \lambda^d(\dif x) \mu(\dif s)\\
    &= 0 \, ,
\end{split}
\end{equation*}
where the last equality comes from the fact that for $(x, s)$ fixed $l(x, s, \gamma) = 0$ only if $1/ \gamma$ is in the spectrum of $\cH(x, s)$ which is finite. Therefore we have a $\Gamma$ a set of full measure in $\bbR_{+}$  such that for $\gamma \in \Gamma$ we have  $\lambda^d \otimes \mu (\{ (x, s): l(x,s,\gamma) = 0\}) = 0$. Once again applying Fubini's theorem we get that for almost every $s \in \Xi$ we have $\{ x:  g_{s, \gamma}(\cdot) \text{ is a local diffeomorphism at } x\})$
is of $\lambda^d$-full measure (since for each $s$, $f(\cdot, x)$ is almost everywhere $C^2$).
Finally, for $A\subset \RR^d$, $\gamma \in \Gamma$ and $\nu \in \cM_{abs}(\bbR^d)$, we have
\begin{equation*}
    \nu \Pg(A) = \nu \otimes \mu(\{ (x, s) : g_{s, \gamma}(x) \in A\})  \leq \lambda^d \otimes \mu(\{ (x, s) : g_{s, \gamma}(x) \in A\})\, ,
\end{equation*}
and by Fubini's theorem,
\begin{equation*}
\begin{split}
  \lambda^d \otimes \mu(\{ (x, s) : g_{s, \gamma}(x) \in A\})
  &= \int \lambda^d(\{ x: g_{s, \gamma}(x) \in A\}) \mu(\dif s) \\
    &= \int \lambda^d(\{ x: g_{s, \gamma}(x) \in A \text { and } f(\cdot, s) \text{ is } C^2 \text{ at } x\}) \mu(\dif s) \\
    &= \int \lambda^d(\{ x: g_{s, \gamma}(x) \in A \text{ and } g_{s, \gamma}(\cdot) \text{ is a local diffeomorphism at } x\}) \mu(\dif s) \, .
\end{split}
\end{equation*}
Now by separability of $\bbR^d$ there is a countable family of open neighborhoods $(V_i)_{i \in \bbN}$ such that for any open set $O$ we have $O = \bigcup_{j \in J}V_j$. The set of $x$ where $g(\cdot, s, \gamma)$ is a local diffeomorphism is an open set, hence
\begin{equation*}
  \{ x: g_{s, \gamma}(x) \in A \text{ and } g_{s, \gamma}(\cdot) \text{ is a local diffeomorphism at } x\} = \bigcup_{i \in I} V_i \cap \{ x: g_{s, \gamma}(x) \in A \} \, .
\end{equation*}
Since an image of a null set by a diffeomorphism is a null set we have
\begin{equation*}
  \lambda^d(\{ x: g_{s, \gamma}(x) \in A \} \cap V_i) = 0 \, .
\end{equation*}
Hence, $\nu \Pg(A) = 0$, which proves our claim.

\subsection{Proof of Theorem~\ref{th:rm}}
\label{sec:proofrm}

Take $\nu \ll \lambda$ and a SGD sequence $(x_n)_{n \in \bbN}$, let $S_1 \subset \RR^d$  be the set of $x$ for which  $\nabla f(x,s)$ exists
for $\mu$- almost every $s$, \emph{i.e.},
\[
S_1 \eqdef \left\{ x \in \RR^d \, : \,
\int_{\Xi} (1 - \1_{\Delta_f}(x,s)) \
    \mu(ds) = 0 \right\} .
\]
When Assumption ~\ref{hyp:model} holds, Rademacher's theorem,
lemma~\ref{nabmes} and Fubini's theorem imply that $S_1 \in
\mcB(\RR^d)$ and $\lambda(\RR^d \setminus S_1) = 0$. Hence, for $\mu$-a.e. $s$ we have $f(\cdot, s)$ differentiable at $x_0$, and since $\xi_1 \sim \mu$, $f(\cdot, \xi_1)$ is differentiable at $x_0$. Now by Rademacher's
theorem again, the set $S_2 \subset \RR^d$ where $F$ is differentiable
satisfies $\lambda(\RR^d \setminus S_2) = 0$, therefore $F$ is differentiable at $x_0$.  Moreover, with probability one $x_0$ is in $S_1 \cap S_2$.  Define $A(x) \eqdef \{s\in \Xi: (x,s)\notin \Delta_f\}$.
By Assumption~\ref{hyp:model}, $\| \nabla f(x,\cdot) \|$ is $\mu$-integrable.
Moreover, for all $x \in S_1 \cap S_2$ and all $v \in \RR^d$
\begin{align*}
\ps{\int \nabla f(x,s)\1_{\Delta_f}(x,s) \, \mu(ds), v}
 &= \int_{\Xi\setminus A(x)} \ps{\nabla f(x,s), v} \, \mu(ds) \\
 &= \int_{\Xi\setminus A(x)} \lim_{t\in \RR^* \to 0}
  \frac{f(x + tv, s) - f(x,s)}{t} \, \mu(ds) \\
 &= \lim_{t\in \RR^* \to 0} \int_{\Xi}
  \frac{f(x + tv, s) - f(x,s)}{t} \, \mu(ds)\\
 &= \lim_{t\in \RR^* \to 0} \frac{F(x + tv) - F(x)}{t} \ = \ps{\nabla F(x), v}
\end{align*}
where the interchange between the limit and the integral follows from
Assumption~\ref{hyp:model} and the dominated convergence theorem. Hence,
$\nabla F(x) = \int \nabla f(x,s)\1_{\Delta_f}(x,s) \, \mu(ds)$ for all $x\in S_1 \cap S_2$.
Now denote by $\nu_n$ the law of $x_n$. Since we assumed that $\nu_0 \ll \lambda$,
it holds that $\PP^{\nu}(x_0 \in S_1 \cap S_2 ) = 1$. Therefore, with probability one,
$$x_1 = x_1\1_{S_1\cap S_2}(x_0) = (x_0 - \gamma\nabla f(x_0,\xi_1))\1_{S_1\cap S_2}(x_0) = x_0 - \gamma \nabla f(x_0, \xi_1) \, .$$
Thus, $x_1$ is integrable whenever $x_0$ is integrable, and
$\EE_0(x_1) = x_0-\gamma \nabla F(x_0)$.
Since by Assumption $\nu_1 \ll \lambda$ we can iterate our argument for $x_2$ and then for all $x_n$
and the conclusions of Theorem~\ref{th:rm} follow.

\subsection{Proof of Theorem~\ref{th:apt}}
\label{sec:proofapt}

We want to apply \cite[Theorem 5.1.]{bia-hac-sal-stochastics19}, and therefore verify its assumptions \cite[Assumption RM]{bia-hac-sal-stochastics19}. In order to fall in its setting we first need to rewrite our kernel in a more appropriate way.
As $\partial F$ takes nonempty compact values, it admits a measurable selection $\varphi(x) \in \partial F(x)$~\cite[Lemma 18.2 and Corollary 18.15]{ali_bor06}. Take $\gamma \in \Gamma$, a SGD sequence $(x^{\gamma}_n)$
and notice that by Theorem~\ref{th:rm} it is $\bbP^{\nu}$ almost surely always in $\cD_F \cap S_1$, where $S_1$ is the set of $x$ where $\nabla f(x,s)$ exists for $\mu$-a.e. $s$. Therefore its Markov kernel can be equivalently defined as:
\begin{equation*}
  \Pg'(x, g) \eqdef \1_{\cD_F \cap S_1}(x) \Pg(x,g) + \1_{(\cD_F \cap S_1)^c}(x)  g(x - \gamma \varphi(x)) \, .
\end{equation*}
Now we can apply \cite[Theorem 5.1.]{bia-hac-sal-stochastics19} with  $h_{\gamma}(s,x) =- (\1_{\cD_F \cap S_1}(x)  \nabla F(x) + \1_{(\cD_F \cap S_1)^c}(x)\varphi(x))$ (note that it is independent of $s$) and we have $h(x,s) \in H(x,s) = H(x) \eqdef -\partial F(x)$. As we show next, \cite[Assumption RM]{bia-hac-sal-stochastics19} now easily follows.\\
   First, it is immediate from the general properties of the Clarke subdifferential that the set-valued map
  $-\partial F$ is proper and uppersemicontinuous with convex and compact values, hence the assumption (iii) of \cite[Assumption RM]{bia-hac-sal-stochastics19}. Assumption (ii) is immediate by the uppersemicontinuity of $-\partial F$.
  Moreover, we obtain from Assumption~\ref{hyp:model-reinf} that there exists a
  constant $K \geq 0$ such that
  \[
  \|\partial F(x) \| \leq K ( 1 + \| x \| ) .
  \]
  Thus, $\cS_{-\partial F}$ is defined on the whole $\RR^d$, and
  $\cS_{-\partial F}$ is closed in $(C(\RR_+, \RR^d), \bs d)$ (see \cite{aub-cel-(livre)84}), hence assumption (v). Finally, assumption (vi) comes from Assumption~\ref{hyp:model-reinf}.

We remark that although, \cite[Theorem 5.1]{bia-hac-sal-stochastics19} deals with a family
of measures $(\bbP^{a})_{a \in \cK}$, the proofs remain unchanged when we consider $(\bbP^{\nu})_{\nu \in \cM_{abs}(\cK)}$.

\subsection{Proof of Theorems~\ref{th:main} and \ref{th:main_proj}}\label{sec:proof_main}

Both theorems are proved in the same way. In the following $Q_{\gamma}$ will denote either $\Pg$ and in this case $\sH$ will denote $- \partial F$, or $Q_{\gamma} = \Pg^r$ and $\sH = - \partial F - \cN_r$.
The proof will be done in three steps:
\begin{itemize}
\item Lemma~\ref{ergodicity}: $Q_{\gamma}$ has a unique invariant probability
distribution $\pi_\gamma$, with $\pi_{\gamma} \in \cM_{abs}$ if $Q_{\gamma} = \Pg$ and $\pi_{\gamma}\in \cM^r$ otherwise, moreover $Q_{\gamma}$ is
ergodic in the sense of the Total Variation norm.
\item Lemma~\ref{pr:PH}: The family
 $\{\pi_\gamma\}_{\gamma\in(0,\gamma_0]}$ is tight.
\item Proposition~\ref{pr:acc_inv_flow}: The accumulation points of
 $\{\pi_\gamma\}_{\gamma\in(0,\gamma_0]}$ as $\gamma\to 0$ are invariant for the DI $\dot{\sx} \in \sH(\sx)$.
\end{itemize}
Before stating Lemma~\ref{ergodicity}, we recall a general result on Markov
processes.  Let $Q : \RR^d \times \mcB(\RR^d) \to [0,1]$ be a Markov kernel on
$\RR^d$.  A set $B \subset \RR^d$ is
said to be a small-set for the kernel $Q$ if there exists a positive measure $\rho$
on $\RR^d$ such that $Q(x,A) \geq \rho(A)$ for each $A\in \mcB(\RR^d)$, $x\in B$.
\begin{proposition}
\label{markov-gal}
Assume that $B$ is a small set for $Q$. Furthermore, assume that there
exists a measurable function $W : \RR^d \to [0,\infty)$ that is defined on
$\RR^d$ and bounded on $B$, and a real number $b \geq 0$, such that
\begin{equation}
\label{drift-general}
QW \leq W - 1 + b \1_B .
\end{equation}
Then, $Q$ admits a unique invariant probability distribution $\pi$, and
moreover, the ergodicity result
\begin{equation}
\label{cvg-tv}
\forall x \in \RR^d, \
  \| Q^n(x,\cdot) - \pi \|_{\text{TV}} \xrightarrow[n\to\infty]{} 0
\end{equation}
holds true.
\end{proposition}
Indeed, by \cite[Theorem~11.3.4]{mey-twe-livre09}, the kernel $Q$ is a so-called
positive Harris recurrent, meaning among others that it has a unique invariant
probability distribution. Moreover, $Q$ is aperiodic, hence the
convergence~\eqref{cvg-tv}, as shown by, \emph{e.g.},
\cite[Theorem~13.0.1]{mey-twe-livre09}.

\begin{lemma}
\label{ergodicity}
Assume that either Assumptions \ref{markov}-\eqref{smallset}
\ref{markov}-\eqref{drift} hold if $Q_{\gamma} = \Pg$ or Assumption~\ref{hyp:smallset_proj} holds and $r \leq R$ if $Q_{\gamma} = \Pg^r$, then for every $\gamma\in (0,\gamma_0]$, the kernel
$Q_{\gamma}$ admits a unique invariant measure $\pi_\gamma$. Moreover,
\begin{equation}
\label{eq:cv-meyne}
\forall x \in \RR^d, \
  \left\| Q_{\gamma}^n(x, \cdot) - \pi_\gamma \right\|_{\text{TV}}
    \xrightarrow[n\to\infty]{} 0 .
\end{equation}
Finally, if $Q_{\gamma} = \Pg$, assumptions of Theorem~\ref{th:rm} hold true and $\gamma \in \Gamma$ then $\pi_\gamma$ is
absolutely continuous w.r.t. the Lebesgue measure. If $Q_{\gamma} = \Pg^r$ and assumptions of Theorem~\ref{th:rm_proj} hold true, then $\pi_{\gamma} \in \cM^r$.
\end{lemma}
\begin{proof}
By the inequality \eqref{eq:drift}, the kernel $\Pg$ satisfies an inequality
of the type~\eqref{drift-general}, namely,
$\Pg V \leq V - \alpha(\gamma) \theta + C\alpha(\gamma) \1_{\| x \| \leq R}$,
for some $\theta, C>0$. Similarly, under Assumption~\ref{hyp:smallset_proj} and $r \leq R$, we have that for every $x \in \cl({B}(0,r))$:
\begin{equation*}
  \Pg^r(x, A) = \Pg(x, \Pi_r^{-1}(A)) \geq \rho_{\gamma}(\Pi_r^{-1}(A)) \, ,
\end{equation*}
that is to say $\cl({B}(0,r))$ is a small set for $\Pg^r$. Inequality of the type Assumption~\ref{markov}-\eqref{drift}--\eqref{tension} then hold for e.g. $C =r$, $\alpha(\gamma) = 1$, $V = \norm{x} + r \1_{\norm{x}> r}$
 and $p = \norm{x}$.

Consider the case where $Q_{\gamma} = \Pg$, to prove that $\pi_\gamma$ is absolutely continuous w.r.t. the Lebesgue measure,
consider a $\lambda$-null set $A$. By the convergence~\eqref{eq:cv-meyne},
we obtain that for any $x \in \bbR^d$, $\Pg^n(x,A) \rightarrow \pi_{\gamma}(A)$.
Now take $\nu \ll \lambda$. By Proposition~\ref{tout-se-vaut}, we have that
$\nu \Pg^n \ll \lambda$. Hence, by the dominated convergence theorem,
\begin{equation*}
0 = \nu \Pg^n(A) =  \int \Pg^n(x,A) \nu(\dif x) \rightarrow \int \pi_{\gamma}(A) \nu(\dif x) = \pi_{\gamma}(A)\,.
\end{equation*}

If $Q_{\gamma} = \Pg^r$ we obtain the same result with the help of Proposition~\ref{tout-se-vaut-proj}.
\end{proof}

\begin{lemma}\label{pr:PH}
Let either Assumptions \ref{markov}-\eqref{smallset} -- \ref{markov}-\eqref{tension}
hold if $Q_{\gamma} = \Pg$ or Assumption~\ref{hyp:smallset_proj} hold and $r \leq R$ if $Q_{\gamma} = \Pg^r$. Let $\pi_{\gamma}$ be the invariant distribution of  $Q_{\gamma}$. Then,
the family $\{ \pi_{\gamma} : \gamma \in (0, \gamma_0 ]\}$ is tight.
\end{lemma}
\begin{proof}
  If $Q_{\gamma} = \Pg^r$ then the family $\pi_{\gamma}$ is supported by $\cl({B}(0,r))$ and is, therefore, tight.
  Otherwise we iterate~\eqref{eq:drift}, to obtain:
  \begin{equation*}
    \sum_{k=0}^n Q_{\gamma}^{k+1} V \leq \sum_{k=0}^n Q_{\gamma}^{k} V - \alpha(\gamma) \sum_{k=0}^n Q_{\gamma}^{k} p + C(n+1)\alpha(\gamma) \, .
  \end{equation*}
  Therefore, since $ 0 \leq Q_{\gamma}^k V < + \infty$ we have:
  \begin{equation*}
    \alpha(\gamma) \sum_{k=0}^n Q_{\gamma}^{k} p \leq V + C(n+1) \alpha(\gamma) \, .
  \end{equation*}
  For a fixed $M > 0$ we will bound now  $\pi_{\gamma} ( p \wedge M)$. Since $\pi_{\gamma}$ is an invariant distribution for $Q_{\gamma}$,
 we have $\pi_\gamma \Pg^k = \pi_\gamma$. Hence, we have:
  \begin{align*}
    \pi_{\gamma}(p \wedge M ) &= \frac{1}{n+1}\sum_{k=0}^n
    \pi_{\gamma} Q_{\gamma}^{k} (p \wedge M) \leq \frac{1}{n+1}\sum_{k=0}^n
    \pi_{\gamma} (Q_{\gamma}^{k} p \wedge M) \\ &\leq
    \pi_{\gamma}\left(\left[\frac{V}{(n +1) \alpha(\gamma)} +
        C\right] \wedge M\right) \, .
  \end{align*}
  Letting $n \rightarrow + \infty$, by the dominated convergence
  theorem we obtain $\pi_{\gamma}(p \wedge M) \leq \pi_{\gamma} ( C
  \wedge M)$. And therefore by monotone convergence theorem $\pi_{\gamma}(p) \leq C$.\\
  Fix now $\varepsilon > 0$, there is a $K >0$ such that $\frac{C}{K}
  \leq \varepsilon$, and by coercivity of $p$ there is $r > 0$ such
  that:
  \begin{equation*}
    \pi_{\gamma}(\norm{x} > r) \leq \pi_{\gamma}(p > K) \leq \frac{C}{K}
  \end{equation*}
  where the last bound comes from Markov's inequality. This concludes the proof.
\end{proof}

The next proposition will show that any accumulation point of $\pi_{\gamma}$ is an invariant measure for the set-valued flow induced by the DI $\dot{\sx}(t) \in \sH(\sx(t))$, first we introduce some definitions. Define the shift operator $\Theta_t :C(\bbR_{+}, \bbR^d)
\rightarrow C(\bbR_{+}, \bbR^d)$ by $\Theta_t (x) = x(t + \cdot)$,
and the projection operator $p_0: C(\bbR_{+}, \bbR^d) \rightarrow
\bbR^d$ by $p_0(x) = x(0)$. Then, we have the following definition (see
\cite{rot-san-13} for details):
\begin{definition}\label{def:inv_flow}
  We say that $\pi \in \cM(\bbR^d)$ is an invariant distribution for the flow induced by the DI $\dot{\sx}(t) \in \sH(\sx(t))$, if there is $\nu \in \cM(C(\bbR_{+}, \bbR^d))$, such that:
  \begin{enumerate}[i)]
    \item $\support \nu \in \overline{\cS_{\sH}(\bbR^d)}$,
    \item $\nu \Theta_t^{-1} = \nu$,
    \item $\nu p_0^{-1} = \pi$.
  \end{enumerate}
\end{definition}

\begin{proposition}\label{pr:acc_inv_flow}
  Let Assumptions ~\ref{hyp:model}--\ref{hyp:zero_in_gamma} and \ref{markov} hold true. Denote by
$\pi_\gamma$ the unique invariant distribution of $\Pg$.
Let $(\gamma_n)$ be a sequence on $(0,\gamma_0]\cap\Gamma$ s.t. $\gamma_n\to 0$ and $\pi_{\gamma_n}$ converges narrowly to some
probability measure $\pi$. Then, $\pi$ is an invariant distribution for the flow induced by $\dot{\sx}(t) \in - \partial F(\sx(t))$.

Similarly, under Assumptions~\ref{hyp:model}--\ref{hyp:model-reinf} and \ref{hyp:proj_ker}--\ref{hyp:smallset_proj}, for $r \leq R$, denoting $\pi_{\gamma}$ the unique invariant distribution of $\Pg^r$, if $\pi_{\gamma_n} \rightarrow \pi$, then $\pi$ is an invariant distribution for the flow induced by $\dot{\sx}(t) \in - \partial F(\sx(t)) - \cN_r(\sx(t))$.
\end{proposition}

\begin{proof}
Consider the case where $Q_{\gamma} = \Pg$.
The proof essentially follows \cite[section 7.]{bia-hac-sal-stochastics19}. Fix an $\varepsilon > 0$ and write $\pi_n$ instead of $\pi_{\gamma_n}$ for simplicity.
By Lemma~\ref{pr:PH} we have a compact $K$ such that $\pi_{n}(K) > 1 - \varepsilon$, we thus can define the conditional measures $\pi_{n}^{K}(A) := \frac{\pi_{n}(A \cap K)}{\pi_{n}(K)}$.
Moreover, we have $\pi_{n}^K \in \cM_{abs}(K)$, therefore we can apply Theorem~\ref{th:apt} and get that there is a compact set $\cC$ of $C(\bbR^{+}, \bbR^d)$ such that $\bbP^{\pi_{\gamma_n}^K, \gamma_n} \sX_{\gamma_n}^{-1}(\cC) \geq 1 - \varepsilon$. Now we have
  \begin{equation*}
    \bbP^{\pi_{n}, \gamma_n}(\cdot) = \int_{\bbR^d}\bbP^{a, \gamma_n} (\cdot) \pi_{n}(\dif a) \geq \int_K \bbP^{a, \gamma_n} (\cdot) \pi_{n}(\dif a)\geq \pi_{n}(K)\bbP^{\pi_n^K, \gamma_n}(\cdot) \, ,
  \end{equation*}
hence
  \begin{equation*}
    \bbP^{\pi_{\gamma_n}, \gamma_n}\sX_{\gamma_n}^{-1}(\cC) \geq \pi_{n}(K) \bbP^{\pi_{\gamma_n}^K, \gamma_n} \sX_{\gamma_n}^{-1}(\cC) \geq (1 - \varepsilon)^2 \, .
  \end{equation*}
  Since $\varepsilon $ is arbitrary this proves the tightness of $v_n:=\bbP^{\pi_{\gamma_n}, \gamma_n}\sX_{\gamma_n}^{-1}$. Take $\pi_n \rightarrow \pi$ and $v_n \rightarrow v \in \cM(C(\bbR_{+}, \bbR^d))$.
 We now prove that $v$ is an invariant distribution for the flow induced by the DI associated to $-\partial F$ (see Definition~\ref{def:inv_flow}.)\\
  We have $\pi_n = v_n p_0^{-1}$, by continuity of $p_0$. Thus, $\pi = v p_0^{-1}$. Therefore, we have (iii) of Definition~\ref{def:inv_flow}.
  Let $\eta > 0$. By weak convergence of $v_n$,
  \begin{equation*}
  \begin{split}
       v(\{x\in C(\bbR_+, \bbR^d) : d(x, \cS_{- \partial F}(\bbR^d)) \leq \eta \}) &\geq \limsup_{n} v_n(\{x\in C(\bbR_+, \bbR^d) : d(x, \cS_{- \partial F}(\bbR^d)) \leq \eta \}) \,
  \end{split}
  \end{equation*}
  and
  \begin{equation*}
  \begin{split}
        v_n(\{x\in C(\bbR_+, \bbR^d) : d(x, \cS_{- \partial F}(\bbR^d)) \leq \eta \}) &\geq v_n( \{x\in C(\bbR_+, \bbR^d) : d(x, \cS_{- \partial F}(K)) < \eta \}) \\
        &\geq \pi_{n}(K)\bbP^{\pi_{\gamma_n}^K, \gamma_n}(d(\sX^{\gamma_n}, \cS_{- \partial F}(K)) < \eta)\\
        &\geq (1 - \varepsilon)\bbP^{\pi_{\gamma_n}^K, \gamma_n}(d(\sX^{\gamma_n}, \cS_{- \partial F}(K)) < \eta) \, .
  \end{split}
  \end{equation*}
  The last term converges to $1- \varepsilon$, by Theorem~\ref{th:apt}, and by weak convergence we have $v(\{x\in C(\bbR_+, \bbR^d) : d(x, \cS_{- \partial F}(\bbR^d)) \geq \eta \}) \geq (1- \varepsilon)$, now letting $\eta \rightarrow 0$,
  by monotone convergence we have $v(\cS_{- \partial F}(\bbR^d))) \geq 1 - \varepsilon$ which proves (i) of Definition~\ref{def:inv_flow}.
  Finally, the second point of Definition~\ref{def:inv_flow} is shown just like in \cite[section 7.]{bia-hac-sal-stochastics19}.

  The proof of the case $Q_{\gamma} = \Pg^r$ is substantially the same under straightforward adaptations.
\end{proof}

After some definitions we recall an important result about the support of a flow-invariant measure.
The limit set $L_f$ of a function $f \in C(\RR_+,\RR^d)$ is
\[
L_f = \bigcap_{t\geq 0} \overline{f([t,\infty))} ,
\]
and the limit set $L_{\cS_{\sH}(a)}$ of a point $a\in\RR^d$ for
$\cS_{\sH}$ is
\[
L_{\cS_{\sH}(a)} = \bigcup_{\sx\in \cS_{\sH}(a)} L_{\sx}.
\]
A point $a \in\RR^d$ is said $\cS_{\sH}$-recurrent if
$a \in L_{\cS_{\sH}(a)}$. The Birkhoff center
$\text{BC}_{\cS_{\sH}}$ of $\cS_{\sH}$ is the closure of the
set of its recurrent points:
\[
\text{BC}_{\cS_{\sH}} =
\overline{
\Bigl\{
a \in \RR^d \ : \ a \in L_{\cS_{\sH}(a)} \Bigr\}} .
\]
In~\cite{fau-rot-13} (see also~\cite{aub-fra-las-91}), a version of
Poincar\'e's recurrence theorem, well-suited for our set-valued evolution
systems, was provided:
\begin{proposition}
\label{poincare}
Each invariant measure for $\cS_{\sH}$ is supported by
$\text{BC}_{\cS_{\sH}}$.
\end{proposition}

With the help of Proposition~\ref{poincare} we can finally prove Theorem~\ref{th:main}.
\begin{proof}

Take $\gamma \in \Gamma$, $\varepsilon > 0$ and $(x^{\gamma}_n)$ an associated SGD sequence.
We have by~\eqref{cvg-tv}:
\begin{equation*}
\limsup_{n\to\infty}
 \PP^{\nu}\left[ \dist( x^{\gamma}_n, \cZ) > \varepsilon \right] = \pi_{\gamma}(\{ x \in \bbR^d : d(x, \cZ) > \varepsilon\}) \, .
\end{equation*}
 Now take any sequence $\gamma_i \rightarrow 0$ with $\gamma_i \in \Gamma$, and $\pi_{\gamma_i}$ the associated invariant distribution, we know from Lemmas~\ref{pr:PH}-\ref{pr:acc_inv_flow} that we can extract a subsequence such that
$\pi_{\gamma_i} \rightarrow \pi$, with $\pi$ an invariant measure for the evolution system $\cS_{- \partial F}$. Therefore by weak convergence we have:
\begin{equation*}
\begin{split}
\lim_{i \rightarrow +\infty} \pi_{\gamma_i}(\{ x \in \bbR^d : d(x, \cZ ) > 2 \varepsilon\}) &\leq \lim_{i \rightarrow + \infty}\pi_{\gamma_i}(\{ x \in \bbR^d : d(x, \cZ ) \geq \varepsilon\}) \\
&\leq \pi(\{ x \in \bbR^d : d(x, \cZ ) \geq \varepsilon\}),
\end{split}
\end{equation*}
where the last line comes from the Portmanteau theorem. We show  that
$\support \pi \subset S$, and therefore the last term is equal to zero, which
concludes the proof. To that end, we make use of Proposition~\ref{poincare},
that shows that each invariant measure of $\cS_{-\partial F}$ is supported by
$\text{BC}_{\cS_{-\partial F}}$. Thus, it remains to show that
$\text{BC}_{\cS_{-\partial F}} = \cZ$ (which at the same time will ensure us that $\cZ$ is nonempty). It is obvious that $\cZ \subset
\text{BC}_{\cS_{-\partial F}}$. To show the reverse inclusion, take
$a \in L_{\cS_{-\partial F}(a)}$. Then, there exists a solution $\sx$ to the
differential inclusion such that  $\sx(0) = a$ and $a \in L_{\sx}$. But under
Assumption~\ref{hyp:chrule} it holds (\cite[lemma 5.2]{dav-dru-kak-lee-19})
that $\| \dot \sx(t) \| = \| \partial_0 F(\sx(t)) \|$ almost everywhere, and,
moreover,
\[
\forall t \geq 0, \quad
 F(\sx(t)) - F(\sx(0)) = - \int_0^t \| \partial_0 F(\sx(u)) \|^2 du .
\]
Therefore $\sx(t) = a$ for each $t \geq 0$,
thus, $a \in S$. Observing that
$\cZ$ is a closed set (since $\partial F$ is graph-closed,
see~\cite[Proposition 2.1.5]{cla-led-ste-wol-livre98}), we obtain that
$\text{BC}_{\cS_{-\partial F}} = \cZ$.

Similarly, take $\gamma_i \rightarrow 0$ and and $(x^{\gamma_i,r}_n)$ the associated  projected SGD sequences. After an extraction we get that $\pi_{\gamma_i} \rightarrow \pi$, with $\pi$ an invariant measure for the flow $\cS_{-\partial F - \cN_r}$ and:

\begin{equation*}
  \lim_{\gamma_i \rightarrow 0}\limsup_{n\to\infty}
   \PP^{\nu}\left[ \dist( x^{\gamma_i, r}_n, \cZ_r) > 2\varepsilon \right] \leq \pi(\{ x \in \bbR^d : d(x, \cZ_r) > \varepsilon\}) \, .
\end{equation*}

Taking $a \in L_{\cS_{- \partial F - \cN_r}(a)}$, and $\sx$ a solution to the associated differential inclusion with $\sx(0) = a$, we get under Assumption~\ref{hyp:chrule} \cite[Lemma 6.3.]{dav-dru-kak-lee-19}  that $\norm{\dot{\sx}(t)} = \min \{ \norm{v}: v \in \partial F(\sx(t)) +\cN_r(\sx(t))\}$, and moreover,

\begin{equation*}
  \forall t \geq 0, \quad F(\sx(t)) - F(\sx(0)) = - \int_0^t \norm{\dot{\sx}(u)}^2 \dif u \, .
\end{equation*}
That is to say $\sx(t) = a$ and $a \in \cZ_r$, which finishes the proof.
\end{proof}
%


\subsection{Proof of Proposition~\ref{additive-ss}}

Denote as $\rho$ the probability distribution of the random variable $\gamma
\epsilon_1$. By assumption, $\rho$ has a continuous density that is positive at
each point of $\RR^d$. We denote as $f$ this density.
Let $\theta_x$ be the probability distribution of the random variable
$Z = x - \gamma\varphi_0(x, \xi_1)$, which is the image of $\mu$ by the
function $x - \gamma\varphi_0(x,\cdot)$. Our purpose is to show that
\[
\exists \varepsilon > 0, \
\forall x \in \cl({B}(0,R)), \ \forall A \in \mcB(\RR^d), \
(\theta_x \otimes \rho) \left[ Z + \gamma\eta_1 \in A \right]
  \geq \varepsilon \,  \lambda( A \cap \cl({B}(0,1)) ) .
\]
Given $L > 0$, we have by Assumption~\ref{hyp:model-reinf} and Markov's
inequality that there exists a constant $K > 0$ such that
\[
\theta_x \left[ Z \not\in \cl({B}(0,L)) \right]
 \leq \frac KL ( 1 + \| x \| ).
\]
Thus, taking $L$ large enough, we obtain that $\forall x \in \cl({B}(0,R)),
\ \theta_x \left[ Z \not\in \cl({B}(0,L)) \right] < 1/2$. Moreover, we
can always choose $\varepsilon > 0$ is such a way that
$f(u) \geq 2 \varepsilon$ for $u \in \cl({B}(0,L+1))$, by the continuity
and the positivity of $f$ on the compact $\cl({B}(0,L+1))$. Thus,
\begin{align*}
(\theta_x \otimes \rho) \left[ Z + \gamma\eta_1 \in A \right] &=
\int_A du \int_{\RR^d} \theta_x(dv) \, f(u-v) \\
&\geq \int_{A\cap\cl({B}(0,1))} du \int_{\cl({B}(0,L))}
  \theta_x(dv) \, f(u-v) \\
&\geq 2 \varepsilon \int_{A\cap\cl({B}(0,1))} du \int_{\cl({B}(0,L))}
  \theta_x(dv) \\
&\geq \varepsilon \, \lambda (A \cap \cl({B}(0,1))) .
\end{align*}

\subsection{Proof of Proposition~\ref{prop:PHsufficient}}

By Lebourg's mean value theorem~\cite[Theorem~2.4]{cla-led-ste-wol-livre98}, for
each $n \in \NN$, there exists $\alpha_n \in [0,1]$ and
$\zeta_n \in \partial F(u_n)$ with
$u_n = x_n - \alpha_n \gamma \nabla f(x_n, \xi_{n+1}) \1_{\Delta_f}(x_n, \xi_{n+1})$, such that
\[
F(x_{n+1}) = F(x_n) -\gamma \ps{\zeta_n, \nabla f(x_n, \xi_{n+1})} \1_{\Delta_f}(x_n, \xi_{n+1}),
\]
and the proof of this theorem (see~\cite[Theorem~2.4]{cla-led-ste-wol-livre98}
again) shows that $u_n$ can be chosen measurably as a function of
$(x_n, \xi_{n+1})$.

In the following, for the ease of readability, we  make use of shorthand
(and abusive) notations of the
type $\1_{\| x \| > 2R} \ps{\nabla F(x), \cdots}$ to refer to
$\ps{\nabla F(x), \cdots}$ if $\| x \| > 2R$ and to zero if not.
We also denote $\nabla f(x_n, \xi_{n+1})$ as $\nabla f_{n+1}$ to shorten the equations.
We write
\begin{align}
F(x_{n+1}) &= F(x_n)
   - \gamma \1_{\| x_n \| \leq 2R} \ps{\zeta_n,\nabla f_{n+1}} \1_{\Delta_f}(x_n, \xi_{n+1}) \nonumber \\
   &\phantom{=}
   - \gamma \1_{\| x_n \| > 2R} \ps{\zeta_n - \nabla F(x_n),
                           \nabla f_{n+1}}
   - \gamma \1_{\| x_n \| > 2R} \ps{\nabla F(x_n), \nabla f_{n+1}} .
\nonumber
\end{align}
We shall prove that
\begin{align}
\EE_n &F(x_{n+1}) \leq F(x_n)
   - \gamma \1_{\| x_n \| > 2R} \| \nabla F(x_n) \|^2 + \gamma K \1_{\| x_n \| \leq 2R}\nonumber \\
&+ \gamma^2 K \1_{\| x_n \| > 2R}
  \left( (1+ \|\nabla F(x_n) \| )
  \Bigl(\int \|\nabla f(x_n, s) \|^2 \, \mu(ds)\Bigr)^{1/2}
 +  \int \| \nabla f(x_n, s) \|^2 \, \mu(ds) \right)
\label{Fdrift}
\end{align}
where the constant $K > 0$ is an absolute finite constant that can change from
line to line in the derivations below. To that end, we write
\begin{align}
F(x_{n+1}) &= F(x_n) - \gamma \1_{\| x_n \| \leq 2R} \1_{\|u_n \| \leq R}
\ps{\zeta_n, \nabla f_{n+1}} \1_{\Delta_f}(x_n, \xi_{n+1}) \nonumber\\
&\phantom{=}
   - \gamma \1_{\| x_n \| \leq 2R} \1_{\| u_n \| > R}
   \ps{\zeta_n, \nabla f_{n+1}} \1_{\Delta_f}(x_n, \xi_{n+1}) \nonumber \\
&\phantom{=}
   - \gamma \1_{\| x_n \| > 2R} \1_{\|u_n\| \leq R}
      \ps{\zeta_n - \nabla F(x_n), \nabla f_{n+1}} \nonumber \\
&\phantom{=}
   - \gamma \1_{\| x_n \| > 2R} \1_{\|u_n\| > R}
      \ps{\nabla F(u_n) - \nabla F(x_n), \nabla f_{n+1}} \nonumber \\
&\phantom{=}
   - \gamma \1_{\| x_n \| > 2R} \ps{\nabla F(x_n), \nabla f_{n+1}}
\label{Fn+1n}
\end{align}
We start with the second term at the right hand side of this inequality.
Noting from Assumption~\ref{hyp:model-reinf} that
\[
\1_{\|u_n \| \leq R} \| \zeta_n \| \leq
\sup_{\| x \| \leq R} \| \partial F(x) \|
\leq \sup_{\| x \| \leq R} \int \| \partial f(x,s) \| \, \mu(ds)
\leq \sup_{\| x \| \leq R} \int \kappa(x,s) \, \mu(ds)
 \leq K ,
\]
we have
\[
\gamma \1_{\| x_n \| \leq 2R} \1_{\|u_n \| \leq R}
   | \ps{\zeta_n, \nabla f (x_n, \xi_{n+1})} | \leq
\gamma K \1_{\| x_n \| \leq 2R} \| \nabla f_{n+1} \| ,
\]
and by integrating with respect to $\xi_{n+1}$ and using
Assumption~\ref{hyp:model-reinf} again, we get that
\begin{equation}
\label{ineq1}
\gamma \1_{\| x_n \| \leq 2R} \EE_n [ \1_{\|u_n \| \leq R}
   | \ps{\zeta_n, \nabla f_{n+1}}\1_{\Delta_f}(x_n, \xi_{n+1}) | ] \leq
\gamma K \1_{\| x_n \| \leq 2R}  .
\end{equation}

Using Assumption~\ref{hyp:model-reinf}, the next term at the right hand side
of~\eqref{Fn+1n} can be bounded as
\begin{align*}
   & \gamma \1_{\| x_n \| \leq 2R} \1_{\| u_n \| > R}
   | \ps{\zeta_n, \nabla f_{n+1}}1_{\Delta_f}(x_n, \xi_{n+1}) | \\
 &\leq
   \gamma \1_{\| x_n \| \leq 2R} \1_{\| u_n \| > R}
   \| \nabla F(u_n) \| \, \| \nabla f_{n+1}  \| \\
&\leq
   \gamma \1_{\| x_n \| \leq 2R}
  K \left( 1 + \| x_n \| + \gamma \|\nabla f_{n+1} \|\right)
      \| \nabla f_{n+1} \|   \\
&\leq
   \gamma K \1_{\| x_n \| \leq 2R}
  \left( 1 + \|\nabla f_{n+1} \| +
     \gamma \|\nabla f_{n+1} \|^2 \right)  ,
\end{align*}
which leads to
\begin{equation}
\label{ineq2}
   \gamma \1_{\| x_n \| \leq 2R} \EE_n [ \1_{\| u_n \| > R}
   | \ps{\zeta_n, \nabla f_{n+1}}1_{\Delta_f}(x_n, \xi_{n+1}) | ] \leq
   \gamma K \1_{\| x_n \| \leq 2R}
\end{equation}
by using Assumption~\ref{hyp:model-reinf}.

We tackle the next term at the right hand side of~\eqref{Fn+1n}.
Fix a $x_\star \not\in \cl (B(0,R))$. By our assumptions it holds
that each $x \not\in \cl(B(0,R))$,
\[
\| \nabla f(x,s) \| \leq \| \nabla f(x_\star,s) \| + \beta(s)\| x - x_\star \|
 \leq \beta'(s) ( 1 + \| x \|),
\]
where $\beta'(\cdot)$ is square integrable thanks to Assumption~\ref{hyp:model-reinf}.
Since
\[
\int \beta'(s)^2 \mu(ds) =
\int_0^\infty \mu[ \beta'(\cdot) \geq \sqrt{t} ] \, dt < \infty,
\]
it holds that $\mu[ \beta'(\cdot) \geq 1/{t} ] = o_{t\to 0}(t^2)$.
Using triangle inequality, we get that
\begin{align*}
\1_{\| x_n \| > 2R} \1_{\|u_n\| \leq R} &=
\1_{\| x_n \| > 2R}
  \1_{\|x_n - \alpha_n \gamma \nabla f_{n+1}\| \leq R}
 \leq \1_{\| x_n \| > 2R}
  \1_{\|\nabla f_{n+1}\| \geq (\| x_n \| - R)/\gamma}  \\
 &\leq \1_{\| x_n \| > 2R}
   \1_{\beta'(\xi_{n+1}) \geq \frac{\|x_n\| - R}{\gamma(1 + \|x_n\|)}}
 \leq \1_{\| x_n \| > 2R}
   \1_{\beta'(\xi_{n+1}) \geq \frac{R}{\gamma(1 + 2R)}} .
\end{align*}
Using this result, we write
\begin{align*}
 \gamma \1_{\| x_n \| > 2R} \1_{\|u_n\| \leq R}
      | \ps{\zeta_n, \nabla f_{n+1}} |
 &\leq K \gamma \1_{\| x_n \| > 2R} \1_{\|u_n\| \leq R}
 \| \nabla f_{n+1} \| \\
 &\leq K \gamma \1_{\| x_n \| > 2R} \| \nabla f_{n+1} \|
  \1_{\beta'(\xi_{n+1}) \geq \frac{R}{\gamma(1 + 2R)}}  \\
\end{align*}
Consequently,
\begin{align}
 \gamma \1_{\| x_n \| > 2R} \EE_n[ \1_{\|u_n\| \leq R}
      | \ps{\zeta_n, \nabla f_{n+1}} | ]
 &\leq
 \gamma K \1_{\| x_n \| > 2R}
  \Bigl( \int \| \nabla f(x_n, s) \|^2 \, \mu(ds) \Bigr)^{1/2}
  \mu[ \beta'(\cdot) \geq K / \gamma ]^{1/2} \nonumber \\
 &\leq \gamma^2 K \1_{\| x_n \| > 2R}
  \Bigl( \int \| \nabla f(x_n, s) \|^2 \, \mu(ds) \Bigr)^{1/2} .
\label{ineq3}
\end{align}

Similarly,
\begin{align*}
 \gamma \1_{\| x_n \| > 2R} \1_{\|u_n\| \leq R}
      | \ps{\nabla F(x_n), \nabla f_{n+1}} |
 &\leq \gamma  K \1_{\| x_n \| > 2R} \| \nabla F(x_n) \| \,
      \| \nabla f_{n+1} \|
  \1_{\beta'(\xi_{n+1}) \geq \frac{R}{\gamma(1 + 2R)}} ,
\end{align*}
thus,
\begin{equation}
\gamma \1_{\| x_n \| > 2R} \EE_n \left[ \1_{\|u_n\| \leq R}
      | \ps{\nabla F(x_n), \nabla f_{n+1}} | \right]
 \leq \gamma^2 K \1_{\| x_n \| > 2R} \| \nabla F(x_n) \|
  \Bigl( \int \| \nabla f(x_n, s) \|^2 \, \mu(ds) \Bigr)^{1/2} .
\label{ineq4}
\end{equation}
We have that $\nabla F$ is Lipschitz outside
$\cl(B(0,R))$. Thus, the next to last term at the right hand side
of~\eqref{Fn+1n} satisfies
\[
  \gamma \1_{\| x_n \| > 2R} \1_{\|u_n\| > R}
 | \ps{\nabla F(u_n) - \nabla F(x_n), \nabla f_{n+1}} |
\leq
 \gamma^2 K \1_{\| x_n \| > 2R}
 \| \nabla f_{n+1} \|^2 ,
\]
and we get that
\begin{equation}
  \gamma \1_{\| x_n \| > 2R} \1_{\|u_n\| > R}
 \EE_n \left[ | \ps{\nabla F(u_n) - \nabla F(x_n), \nabla f_{n+1}} | \right]
 \leq
 \gamma^2 K \1_{\| x_n \| > 2R}
  \int \| \nabla f(x_n, s) \|^2 \mu(ds) .
\label{ineq5}
\end{equation}

Finally, we have
\begin{equation}
 - \gamma \1_{\| x_n \| > 2R}
  \EE_n \left[\ps{\nabla F(x_n), \nabla f_{n+1}}\right] =
 - \gamma \1_{\| x_n \| > 2R} \| \nabla F(x_n) \|^2 .
\label{ineq6}
\end{equation}
Inequalities~\eqref{ineq1}--\eqref{ineq6} lead to~\eqref{Fdrift}.

Using Assumption (iii) of Proposition~\ref{prop:PHsufficient},
Inequality~\eqref{Fdrift} leads to Inequality~\eqref{Fdrift2}.
The validity of Assumptions~\ref{markov}-\eqref{drift} and
\ref{markov}-\eqref{tension} can then be checked easily.

\subsection{Proof of Proposition~\ref{tout-se-vaut-proj}}

The next Lemma is the key ingredient in the proofs of
Section~\ref{sec:proj-sgd}.

\begin{lemma}\label{lm:good_r}
Assume that $f(\cdot, s)$ is locally Lipschitz continuous for every $s \in
\Xi$. Then for $\lambda^1 \otimes \lambda^d \otimes \mu$-almost all
$(r, x, s)$ with $r > 0$, it holds that $(\Pi_r(x),s) \in \Delta_f$. For
$\lambda^1 \otimes \lambda^d$-almost all $(r,x)$ with $r > 0$, it holds
that $\Pi_r(x)\in\cD_F$.
\end{lemma}

\begin{proof}

  Our first aim is to show that
  \begin{equation}\label{eq:proj_diff}
    \int \1_{\Delta_f^c}(\Pi_r(x),s) \,
   \lambda^1(\dif r) \otimes \lambda^d(\dif x) \otimes \mu(\dif s) = 0 \, .
  \end{equation}
  First, note by Fubini's theorem that
  \begin{equation}\label{eq:pol_coord_proj}
  0 = \int \1_{\Delta_f^c}(x,s) \, \lambda^d(\dif x)\otimes\mu(\dif s)
   = \int_{\Xi\times \RR_+} \int_{\SSS(1)}
   \1_{\Delta_f^c}(r \theta,s) r^{d-1} \varrho_1(\dif \theta) \
    \mu \otimes \lambda^1(\dif s \times \dif r) \, ,
  \end{equation}
that is to say, $\varrho(\{ \theta : (r \theta, s) \in \Delta_f\}) = 0$
for $\mu \otimes \lambda^1$ almost every $(s, r)$ with $r > 0$.
Decompose Equation~\eqref{eq:proj_diff} as
  \begin{multline*}
    \int \1_{\Delta_f^c}(\Pi_r(x), s) \
  \lambda^1(\dif r) \otimes \lambda^d(\dif x) \otimes \mu(\dif s) \\
 = \int  \1_{\norm{x} \geq r}\1_{\Delta_f^c}(\Pi_r(x), s)  \
  \lambda^1(\dif r) \otimes \lambda^d(\dif x) \otimes \mu(\dif s)
+  \int \1_{\norm{x} < r}\1_{\Delta_f^c}(x, s) \
  \lambda^1(\dif r) \otimes \lambda^d(\dif x) \otimes \mu(\dif s) .
  \end{multline*}
Since for each $s$, $f(\cdot, s)$ is differentiable almost everywhere, we
have by Fubini's theorem:
  \begin{equation*}
    \int \1_{\norm{x} < r}\1_{\Delta_f^c}(x, s) \
  \lambda^1(\dif r) \otimes \lambda^d(\dif x) \otimes \mu(\dif s) = 0 .
  \end{equation*}
  Similarly,
  \begin{align*}
    &  \int  \1_{\norm{x} \geq r}\1_{\Delta_f^c}(\Pi_r(x), s)  \
  \lambda^1(\dif r) \otimes \lambda^d(\dif x) \otimes \mu(\dif s) \\
  &=\int \1_{\norm{x} \geq r}\1_{\Delta_f^c}\Bigl(\frac{rx}{\norm{x}},s\Bigr)
  \ \lambda^1(\dif r) \otimes \lambda^d(\dif x) \otimes \mu(\dif s) \\
    &= \int_{\RR_+} \int_{\Xi\times\RR_+} \int_{\SSS(1)} \1_{r' \geq r}
    \1_{\Delta_f^c}(r' \theta,s)(r')^{d-1}
   \varrho(\dif \theta) \
        \mu \otimes \lambda^1(\dif s \times \dif r) \ \lambda^1(\dif r') \\
        &= 0 \, ,
  \end{align*}
with the last equality coming from Equation~\eqref{eq:pol_coord_proj}.
Hence~\eqref{eq:proj_diff}. The second statement can be proven along
similar lines.
\end{proof}

Consider $r> 0$ such that the conclusion of Lemma~\ref{lm:good_r} hold. Then the
almost sure equality of all projected SGD sequence is proven in the same way as
in Proposition~\ref{tout-se-vaut}. We can therefore consider the lazy projected
SGD sequence $x_{n+1}^{\gamma,r} = \Pi_r(x_n^{\gamma,r} - \gamma
\varphi_0(x_n^{\gamma,r}, \xi_{n+1}))$. By Assumption~\ref{hyp:proj_ker} the
law of $x_{n+1/2}^{\gamma,r} \eqdef x_n^{\gamma,r} - \gamma
\varphi_0(x_n^{\gamma,r}, \xi_{n+1})$ is Lebesgue-absolutely continuous. Take
$A$ a borel set of $\bbR^d$ such that $\lambda(A) = \varrho_r(A) = 0$. Then

\begin{equation*}
  \bbP(x_{n+1}^{\gamma,r} \in A) \leq \bbP(x_{n+1/2}^{\gamma,r} \in A) + \bbP\left( r \frac{x_{n+1/2}^{\gamma,r}}{\norm{x_{n+1/2}^{\gamma,r}}} \in A\right) \, .
\end{equation*}
The first term is equal to zero by Lebesgue-absolutely continuity of the law
of $x_{n+1/2}^{\gamma,r}$. For the second term we write:

\begin{equation*}
 \bbP\left( r \frac{x_{n+1/2}^{\gamma,r}}{\norm{x_{n+1/2}^{\gamma,r}}} \in
   A\right)  =
 \int (r')^{d-1}\1_A(r\theta) \varrho(\dif \theta) \lambda^1(\dif r')
  = \int (r')^{d-1}\varrho_r(A) \lambda^1(\dif r') = 0 \, ,
\end{equation*}
which finishes the proof.


\subsection{Proof of Theorems~\ref{th:rm_proj} and~\ref{th:apt_proj}}

Noting that the law of $x_n^{\gamma,r} - \gamma \varphi_0(x_n^{\gamma,r}, \xi_{n+1})$ is Lebesgue-absolutely continuous by Assumption~\ref{hyp:proj_ker}, the first point of Theorem~\ref{th:rm_proj} comes from Lemma~\ref{lm:good_r}. The second point comes upon noticing that $\Pi_r(x) - x \in - \cN_r(\Pi_r(x))$.

Theorem~\ref{th:apt_proj} is proved in the same way as Theorem~\ref{th:apt}, by applying \cite[Theorem 5.1.]{bia-hac-sal-stochastics19} with $h(s, x) = - \nabla F(x) - 1/\gamma (x - \gamma \nabla f(x,s) - \Pi_r(x - \gamma \nabla f(x,s))) \in - \nabla F(x) - \cN_r(x - \gamma \nabla f(x,s))$ and $H(x) = H(s,x) = - \partial F(x) - \cN_r(x)$.


\section*{Acknowledgements}
  The authors wish to thank J\'er\^ome Bolte and Edouard Pauwels for their
  inspiring remarks. This work is partially supported by the Région Ile-de-France.

\bibliographystyle{plain}
\bibliography{math}

\end{document}